\documentclass[a4paper, reqno]{amsart}
\usepackage[margin=1in]{geometry}
\usepackage{enumerate}
\usepackage{amsthm}
\usepackage{amsmath}
\usepackage[mathscr]{euscript}
\usepackage{amssymb}
\usepackage[noadjust]{cite}
\usepackage{comment}
\usepackage{color}
\usepackage{filecontents}
\usepackage{tabularx}

\usepackage{tikz}
\usetikzlibrary{calc,positioning,fit,backgrounds}
\pgfdeclarelayer{background}
\pgfsetlayers{background,main}
\usepackage{caption}

\usepackage{color} 
\usepackage{graphicx}

\usepackage{amssymb,amsthm,amsmath,mathrsfs}
\usepackage{graphicx}
\usepackage{multicol}
\usepackage{color}
\usepackage{tikz}
\usepackage{tikz-cd}
\usepackage[hypertexnames=false]{hyperref}
\hypersetup{
	colorlinks=true,
	linkcolor=blue,
}

\newtheorem*{genericthm*}{\thistheoremname}
\newcommand{\thistheoremname}{???}
\newcounter{genericthm}

\newenvironment{namedthm*}[1]
{\renewcommand{\thistheoremname}{#1}%
	\refstepcounter{genericthm}%
	\begin{genericthm*}}
	{\end{genericthm*}}

\newtheorem{thm}{Theorem}[section]
\newtheorem{lem}[thm]{Lemma}
\newtheorem{prop}[thm]{Proposition}
\newtheorem{cor}[thm]{Corollary}

\theoremstyle{definition}

\newtheorem{defn}[thm]{Definition}

\theoremstyle{remark}

\newtheorem{rem}[thm]{Remark}

\begin{document}
\title{Hilbert transforms on graph products of finite von Neumann algebras}

\author[Lu and Xia]{Xiao-Qi Lu \and Runlian Xia}
	
\address[Xiao-Qi Lu]{School of Mathematics and Statistics, University of Glasgow, University Place, Glasgow G12 8QQ, UK}
\email{Xiaoqi.Lu@glasgow.ac.uk}

\address[Runlian Xia]{School of Mathematics and Statistics, University of Glasgow, University Place, Glasgow G12 8QQ, UK}
\email{Runlian.Xia@glasgow.ac.uk}

\keywords{Hilbert transforms, graph products of von Neumann algebras, noncommutative Cotlar identities, Haagerup-type inequality}

\subjclass[2020]{Primary: 46L07, 46L10, 46L52}

\begin{abstract}
 We study Hilbert transforms on graph products of finite von Neumann algebras, with particular interest in their boundedness on the associated noncommutative $L_p$-spaces for $1<p<\infty$. We establish a generalized Cotlar identity for Hilbert transforms, valid on operators whose lengths exceed a constant depending only on the underlying graph. We further prove that graph products of finite von Neumann algebras satisfying a Haagerup-type inequality admit $L_p$-bounded Hilbert transforms, therefore extending the corresponding result of Mei and Ricard for free products of finite von Neumann algebras. In addition, we obtain several equivalent characterizations of this Haagerup-type inequality.
Our results apply, in particular, to graph products of finite groups, right-angled Hecke von Neumann algebras, and graph products of finite quantum groups. As an application, we provide  a complete characterization of a compactness problem posed by Ozawa for  graph products of finite groups and right-angled Hecke von Neumann algebras.
\end{abstract}

\maketitle

\section{Introduction}


The study of mapping properties (boundedness) of Hilbert transforms is one of the central problems in harmonic analysis. In the classical setting, the Hilbert transform on the real line $\mathbb{R}$ may be viewed as the Fourier multiplier with symbol $-i\,\mathrm{sgn}$, namely $(Hf)^\wedge(\xi)=-i\,\mathrm{sgn}(\xi)\,\widehat{f}(\xi)$,
see \cite{Duoan2001Book}. Equivalently, $H$ may be formally written as
\[
H
=
-iP_+ + iP_-,
\]
where $P_+$ and $P_-$ denote the Fourier multipliers with symbols
$\chi_{\{\xi \in \mathbb{R} : \xi >0\}}$  and $\chi_{\{\xi \in \mathbb{R} : \xi <0\}}$,
respectively. Observe that $P_+$ and $P_-$ are orthogonal projections on $L_2(\mathbb{R})$. In 1924, M.~Riesz established the $L_p$-boundedness of the Hilbert transform for all $1<p<\infty$ using methods from complex analysis \cite{Riesz1924Conjugate,Riesz1928ConjugateMZ}. Later, in 1955, Cotlar gave an alternative algebraic proof \cite{cotlar1955unified}. More precisely, he showed that $H$ satisfies the identity
\begin{equation}
\label{eq:ClassicalCotlar}
\tag{Classical Cotlar}
(Hf)^2
=
2H(f \cdot Hf) + f^2.
\end{equation}
Note that for the Hilbert transform, we have $H^2=-\rm{id} $ whenever it is well-defined. Therefore, the Cotlar identity above can be rewritten as
\begin{equation*}
  \big( H f \big)^2 \, = \, 2 H \big( f \cdot H f \big) -  H \big( H (f^2) \big).
\end{equation*}
Using this identity, one obtains recursively the $L_p$-boundedness for $p=2^k$, starting from the trivial case $p=2$. The full range $1<p<\infty$ then follows by duality and the Riesz-Thorin interpolation theorem.

Noncommutative analogues of the Hilbert transform and the Cotlar identity were introduced by Mei and Ricard in \cite{mei2017free} in the setting of Voiculescu's amalgamated free products of von Neumann algebras. Their construction is defined as an unconditional sum of projections onto reduced operators with prescribed initial syllables. Subsequently, in \cite{gonzalez2022non-commutative}, Gonz\'alez-P\'erez, Parcet, and the second-named author extended these results to von Neumann algebras associated with groups acting on tree-like structures, revealing a connection between noncommutative Cotlar identities and Bass-Serre theory in geometric group theory. More recently, the authors of the present paper investigated analogous questions for groups acting on buildings in \cite{lu2025hilbert}.

In this paper, we study Hilbert transforms in a different direction, namely in the setting of graph products of von Neumann algebras, which simultaneously generalize free products and tensor products of  von Neumann algebras. Graph products of operator algebras were introduced by Caspers and Fima in \cite{caspers2017graph}, and have since become an active area of research. A number of structural and approximation properties for graph products of operator algebras have been established, including Khintchine inequalities \cite{caspers2021graph}, the completely contractive approximation property (CCAP) \cite{BORST2024110350}, and rigidity results \cite{CasChen25, ChiDaDri25, ChiDaDri252,borst2024rigid, HorIoa25}. In free probability, graph products of $*$-algebras are also crucial in the study of $\Lambda$-freeness, see \cite{mlotkowski2004lambda}, where this notion was introduced by M{\l}otkowski. 

Among their many applications, the mapping properties of Hilbert
transforms play a fundamental role in the $L_p$-convergence theory of
partial Fourier series and Fourier integrals on Euclidean spaces. In
fact, by elementary arguments, the $L_p$-boundedness of the Hilbert
transform implies that the frequency truncations of a function
$f\in L_p(\mathbb{R})$ to the intervals $[-N,N]$, for
$N\in\mathbb{N}$, converge to $f$ in the $L_p$-norm as $N\to\infty$.
These truncation operators form an approximate identity on
$L_p(\mathbb{R})$.
In higher dimensions, the interval $[-N,N]$ may be replaced by
$N\cdot\mathcal{P}$, where
$N\cdot\mathcal{P}\subseteq\mathbb{R}^n$ denotes the dilation of a
convex polyhedron $\mathcal{P}$ containing the origin. See \cite{Alfonseca2003, AlfonSoriaVargas2003, Bateman2009kakeya, ParcetRogers2015Directions} for results and related open questions in this direction. 
Consequently, the study of Hilbert transforms on general von Neumann algebras may
lead to applications to approximation properties in noncommutative
$L_p$-spaces. For instance, the Hilbert transforms obtained in
\cite{lu2025hilbert} yield approximate identities on the
noncommutative $L_p$-spaces associated with certain affine Coxeter groups.

On the other hand, the  mapping properties of Hilbert transforms on von Neumann algebras may also provide new tools for studying rigidity and structural results in operator algebras. In recent work, Houdayer and Ioana \cite{houdayer2024asymptotic} used the $L_p$-boundedness of free Hilbert transforms established in \cite{mei2017free} to prove asymptotic freeness in tracial ultraproduct von Neumann algebras. They also constructed the first example of a $\mathrm{II}_1$ factor without property Gamma that is not elementarily equivalent to any free product of diffuse tracial von Neumann algebras. These developments further motivate our study of the $L_p$-boundedness of Hilbert transforms on graph products of von Neumann algebras.

Moreover, at the end of \cite{ozawa2010comment}, Ozawa posed a
compactness problem for free group factors and observed that the
boundedness of the Hilbert transform on
$L_4(\mathcal{L}\mathbb{F}_n)$ would yield a positive solution to the
problem. This question was later resolved in \cite{mei2017free}.
In Section~\ref{sec; Ozawa-application}, we provide  a complete characterization of this question for graph products of finite groups and right-angled Hecke von Neumann algebras.
 In particular, our results provide
new examples beyond the free group factors originally considered by Ozawa.


Compared to the free product case,  the study of Hilbert transforms on graph products of von Neumann algebras is  more delicate, since reduced operators may admit several equivalent reduced expressions. In particular, the initial syllable of a reduced operator is no longer uniquely determined.  To overcome this difficulty, we introduce a natural analogue of the free Hilbert transform in the setting of graph products, and show that it satisfies a generalized Cotlar identity. Our motivating example is  the von Neumann algebra associated with a graph product of groups.

\subsection*{Graph products of groups}

Graph products of groups (introduced by Green in \cite{green1990graph}) are closely related to right-angled Coxeter groups.
 A right-angled Coxeter group can be defined by its \textit{presenting graph}, which is a simplicial graph $\Gamma = (V\Gamma, E\Gamma)$  with vertex set $V\Gamma$ and edge set $E\Gamma$.
The associated \emph{right-angled Coxeter group $W_{\Gamma}$} is defined by 
$$W_{\Gamma} := \langle s\in V\Gamma: s^2 = e,\;  (st)^2 = e\; \text{if}\; \{s,t\}\in E\Gamma \rangle.$$ 
The relations $(st)^{2} = e$ for $\{s,t\}\in E\Gamma$ 
are referred to as \textit{braid relations}.  

For any nontrivial element $ w \in W_{\Gamma}$, a \textit{reduced expression} is a word $s_1\ldots s_n$ with $n\geq 1$, of minimum length such that $w = s_1\ldots s_n$ and $s_i\ne s_{i+1}$ for all $i= 1,\ldots, n-1$. 
Reduced expressions need not be unique, due to the braid relations. Let $\mathcal{W}_{\Gamma}$ denote the set of all reduced expressions of  elements in $W_{\Gamma}$, whose elements we call \textit{reduced words}.
 We write $\textbf{w} \simeq \textbf{v}$ if $\textbf{w}\in \mathcal{W}_{\Gamma}$
can be transformed into $\textbf{v}\in \mathcal{W}_{\Gamma}$ via the braid relations. In this case, they represent the same group element. 

To each vertex $s \in V\Gamma$ associate a nontrivial discrete group $G_s$. The \emph{graph product of the family $\{G_s : s \in V\Gamma\}$} is the discrete  group 
    $$G_{\Gamma} := (\underset{s\in V\Gamma}{\ast}G_s)/N,$$
    where $N$ is the normal subgroup generated by the commutators $\{g_tg_sg_t^{-1}g_s^{-1}: g_t\in G_t, g_s\in G_s\; \text{and}\; \{t,s\}\in E\Gamma\}$. 
By construction, the subgroups $G_s$ and $G_t$ commute in $G_\Gamma$ if and only if $\{s,t\} \in E\Gamma$.
Every nontrivial element $g \in G_\Gamma$ admits a \emph{reduced expression}
\[
g=g_1g_2\cdots g_n,
\]
 where \(g_i\in G_{s_i}\setminus\{e\}\) for $i = 1,\dots,n$, and  $s_{1}\cdots s_{n}$ is a reduced word which lies in $\mathcal{W}_\Gamma$.  
In this case, we say that $g$ has \emph{type} $s_{1}\cdots s_{n}$. Two reduced expressions  of $g$ are \emph{equivalent} if and only if they have equivalent types. The integer $n$ is called the \emph{block length} of $g$, denoted by $\ell(g)$. 
Define the \emph{set of initial
syllables} of $g$ by
\[
\operatorname{Init}(g)
=
\{g_i\in G_{s_i}\setminus\{e\} : 1\leq i\leq n, \{s_j,  s_i \}\in E\Gamma \text{ for every }j<i. \}.
\]
Equivalently, \(\operatorname{Init}(g)\) is the set of syllables that may occur as
the first syllable of some reduced expression of $g$. 
We define the \emph{maximal clique prefix} of \(g\) to be 
\[
\mathrm{MCP}(g)=\prod_{h\in \operatorname{Init}(g)} h.
\]
The factors in the product above commute pairwise, so the product is well-defined. 


If $G_s = \mathbb{Z}_2$ for all $s \in V\Gamma$, then $G_\Gamma = W_\Gamma$ is the right-angled Coxeter group associated with $\Gamma$.
When the graph $\Gamma$ has no edges, one has $
G_\Gamma= * _{s\in V\Gamma}G_s,
$
the free product of the groups $G_s$. If $\Gamma$ is complete, that is,  every two vertices are connected by an edge, then
$
G_\Gamma= \times _{s\in V\Gamma}G_s,
$
the direct product of the groups $G_s$.
Thus, graph products of groups simultaneously generalize free products and direct products.



\subsection*{Hilbert transforms on graph products of groups} 
For a discrete group $G$, its group von Neumann algebra is defined by
$$
  \mathcal{L}G
:= \overline{\mathrm{span}}^{\,w^*}\{\lambda_g : g \in G\}
\subseteq B(\ell_2 G).
$$
The associated noncommutative $L_p$-spaces $L_p(\mathcal{L}G)$ were introduced in \cite{Terp1981lp, pisier2003non}. When $G$ is abelian, these spaces recover the classical $L_p$-spaces over the Pontryagin dual of $G$. Consequently, many problems from classical Fourier analysis admit natural noncommutative analogues in this setting. 

Now let $G = G_\Gamma$, a graph product of groups associated with a simplicial graph $\Gamma$, and let $\varphi_{\Gamma}$ be the associated normal faithful tracial state. 
For each $s \in V\Gamma$, let $\mathscr{L}_s$ denote the subset of $G_\Gamma$ consisting of elements whose equivalent reduced expressions all have types beginning with $s$.
 In particular, when $G_\Gamma$ is a free product of $G_s$,  
$\mathscr{L}_s$  consists of elements whose unique reduced expressions begin with a syllable in $G_s$; when $G_\Gamma$ is a direct product of $G_s$, $\mathscr{L}_s=G_s\setminus \{e\}$ for every $s$; when $G_\Gamma=W_\Gamma$ with 
\begin{center}
\begin{tikzpicture}
		\coordinate[label = below: $s_1$] (s) at (0,0);
		\coordinate[label = below: $s_2$] (t) at (1,0);
		\coordinate[label = below: $s_3$] (r) at (2,0);
		
		\node at (s)[circle, fill, inner sep=1pt]{};
		\node at (t)[circle, fill, inner sep=1pt]{};
		\node at (r)[circle, fill, inner sep=1pt]{};

		 \draw [black] (-1,0) coordinate (l1) node {$\Gamma = $};
		\draw (s)--(t) node[above, midway]{};
	\end{tikzpicture}
    \end{center}
 $\mathscr{L}_{s_1}$ (resp. $\mathscr{L}_{s_2}$) consists of $G_{s_1}\setminus \{e\}$ (resp. $G_{s_2}\setminus \{e\}$) together with all elements whose reduced expressions have types beginning with  $s_1s_3$ (resp. $s_2s_3$), and $\mathscr{L}_{s_3}$ consists of elements whose reduced expressions have types beginning with $s_3$. 
 For each clique $\Gamma_0\in \mathrm{Cliq}(\Gamma,\ge 2)$ (see Definition \ref{defn:clique}(iii)), define $\mathscr{L}_{\Gamma_0} := \{g\in G_{\Gamma}: \mathrm{MCP}(g) \text{ has type } \prod_{s\in V\Gamma_0} s\}$.
 Let $\widehat{f}= \sum_{g \in G_\Gamma} \widehat{f}(g)\delta_g$ be a finitely supported function on $G_\Gamma$. Its associated noncommutative Fourier series  $f= \sum_{g \in G_\Gamma} \widehat{f}(g)\lambda_g$ defines an element of $L_p(\mathcal{L}G_\Gamma)$. For such $f$, define
\[\mathcal{L}_s f:= \sum_{g \in \mathscr{L}_s} \widehat{f}(g)\lambda_g, \quad \text{ and } \quad \mathcal{L}_{\Gamma_0} f:= \sum_{g \in \mathscr{L}_{\Gamma_0}} \widehat{f}(g)\lambda_g.\]
It is immediate that the operators $\mathcal{L}_s$ and $\mathcal{L}_{\Gamma_0}$ are orthogonal projections on $L_2(\mathcal{L}G_\Gamma)$. The natural questions are whether these projections extend to bounded maps on $L_p(\mathcal{L}G_\Gamma)$ for $1 < p < \infty$, and whether the decomposition $G_\Gamma=\{e\}\,\cup\, \bigcup_{\Gamma_0 \in \mathrm{Cliq}(\Gamma,\ge 2)}\mathscr{L}_{\Gamma_0} \,\cup\, \bigcup_{s \in V\Gamma} \mathscr{L}_s$ with respect to the maximal clique prefix
is unconditional. With this in mind, and motivated by the classical Hilbert transform, we define its analogue on graph products of groups by \[H_\varepsilon:=  \varepsilon_0 \varphi_\Gamma + \sum_{s \in V\Gamma} \varepsilon_s \mathcal{L}_s+ \sum_{\Gamma_0\in \mathrm{Cliq}(\Gamma,\ge2)} \varepsilon_{\Gamma_0} \mathcal{L}_{\Gamma_0} ,\]
where $\varepsilon_0, \varepsilon_{\Gamma_0}, \varepsilon_s \in \{\pm1\}$ for all $s \in V\Gamma$ and $\varepsilon=(\varepsilon_0,\varepsilon_{\Gamma_0}, \varepsilon_s)$. 

The Hilbert transforms introduced above also admit other
variations. Let $d\geq 1$ and let $h\in G_\Gamma$ be an element such
that $\ell(h)=d$. Denote by $\mathscr{L}_h$ the subset of $G_\Gamma$
consisting of those elements $a$ such that, whenever $a=a'a''$
is a factorization such that 
$\ell(a)=\ell(a')+\ell(a'')$ with $\ell(a')=d$, one necessarily has
$a'=h$. Define $\mathscr{L}_{C,d}:= \{g\in G_\Gamma : \ell(g)\ge d\}\setminus \bigcup_{\ell(h)=d} \mathscr{L}_h$. We then define the projections
\[
\mathcal{L}_h f
=
\sum_{g\in\mathscr{L}_h}\widehat{f}(g)\lambda_g,
\qquad \text{ and } \quad 
\mathcal{L}_{C,d}f
=
\sum_{g\in\mathscr{L}_{C,d}}\widehat{f}(g)\lambda_g .
\]
The Hilbert transform on graph products of groups depending on the
first $d$ blocks is defined by
$$
H_{\varepsilon}^{\mathcal{L},d} := \varepsilon_0 P_{\le d-1} + \sum_{\ell(h)=d} \varepsilon_h\mathcal{L}_h + {\varepsilon}_d \mathcal{L}_{C,d},
$$
where $\varepsilon = (\varepsilon_0,\varepsilon_h, {\varepsilon}_d)_{\ell(h)=d}$ with $\varepsilon_0, \varepsilon_h, {\varepsilon}_d\in \{\pm1\}$.

The main goal of this paper is to study the boundedness of
$H_\varepsilon$, and more generally of $H_\varepsilon^{\mathcal{L}, d}$, for graph
products of von Neumann algebras, which include graph products of
groups as a special case.

\subsection*{Our main results}
The main tool in the existing literature for establishing the boundedness of Hilbert transforms on von Neumann algebras is the noncommutative Cotlar identity, introduced in \cite{mei2017free} for amalgamated free products of finite von Neumann algebras. In that setting, the identity holds modulo the amalgamated von Neumann subalgebra, while the remaining part is controlled using properties of the associated conditional expectation. Variants of this strategy have also been applied to von Neumann algebras associated with groups acting on suitable geometric structures; see \cite{gonzalez2022non-commutative,lu2025hilbert}. In this setting, the Cotlar identity holds outside a subgroup arising as the stabilizer of a fixed vertex or half-space.
 However, this approach cannot be applied to graph products, mainly due to the presence of commutation relations. Instead, our key observation in this paper is that, in the graph product setting, a suitable noncommutative Cotlar identity emerges only after restricting to operators whose  length exceeds a fixed constant determined by the commutation structure of the graph product of von Neumann algebras $\mathcal{M}_{\Gamma}$.

We assume that the graph $\Gamma$ satisfies the following condition:
\begin{equation}\label{eq; Fi1}\tag{Fi}
\text{there exists } N \ge 0 \text{ such that } |\mathrm{Link}(s)| \le N\quad \text{for all } s \in V\Gamma. 
\end{equation}
Under this assumption on the graph, and subject to a compatibility condition on the coefficients of the Hilbert transform $H_\varepsilon$ defined in Definition \ref{defn;Hilbert transform} (see \eqref{eq: compatibility} for details), $H_\varepsilon$ satisfies a generalized Cotlar identity involving the projection $P_{>6N}$ from $\mathcal{M}_{\Gamma}$ onto the $L_2$-closed subspace generated by reduced operators of length strictly greater than $6N$. Define $H_{\varepsilon}^{\mathrm{op}}(x):= H_{\varepsilon}(x^*)^*$,  for $x \in \mathcal{M}_{\Gamma}$.

\begin{namedthm*}{Theorem A}\label{Thm A}
Let $\Gamma$ be a simplicial graph satisfying~\eqref{eq; Fi1}, let $x,y \in \mathcal{M}_{\Gamma}$, and let $\varepsilon,\varepsilon' \in \{\pm1\}^{\#\mathrm{Cliq}(\Gamma)}$. Assume, moreover, that the Hilbert transform $H_\varepsilon$ and $H^{\mathrm{op}}_{\varepsilon'}$ defined in Definition \ref{defn;Hilbert transform} satisfy the compatibility condition \eqref{eq: compatibility}.
Then
\begin{equation}\label{eq;Cotlar-ab}\tag{Cotlar}
		P_{> 6N}\big(H_{\varepsilon}(x)H_{\varepsilon'}^{\mathrm{op}}(y^*)\big) = P_{> 6N}\Big(	H_{\varepsilon}\big(x	H_{\varepsilon'}^{\mathrm{op}}(y^*)\big) + 	H_{\varepsilon'}^{\mathrm{op}}\big(	H_{\varepsilon}(x)y^*\big)- 	H_{\varepsilon'}^{\mathrm{op}}	H_{\varepsilon}(xy^*)\Big).
     \end{equation}  
\end{namedthm*}

Once we show that  $H_\varepsilon$ satisfies~\eqref{eq;Cotlar-ab}, the remaining difficulty in proving its boundedness on $L_p(\mathcal{M}_{\Gamma})$ is to control its restriction to $P_{\le 6N}(\mathcal{M}_{\Gamma})$. We overcome this difficulty by establishing a new connection to a Haagerup-type inequality for the underlying von Neumann algebra.

As we will see in Theorem~\ref{thm;Hilbert-trans-finite-vnA}, this can be achieved under the assumption that $P_{\le 6N}$ is ultracontractive, that is, it is bounded from $L_2(\mathcal{M}_{\Gamma})$ to $\mathcal{M}_{\Gamma}$.
In Section~\ref{sec; Haagerup-type-inequality}, we explain that this hypothesis is closely related to the Haagerup inequality introduced by Haagerup for free groups \cite{haagerup1978example}, which asserts that the operator norm of elements supported on group elements of a fixed word length is controlled by a polynomial in the length times the $L_2$-norm. This type of estimate was later extended to more general groups by Jolissaint in \cite{jolissaint1990rapidly}, who formalized it as the property of rapid decay (property (RD)); see Subsection~\ref{subsection;group vnA} for further discussion.

In our setting, we will show in Proposition \ref{thm;equivalence-Pd-Haagerup-type} that the assumption that $P_{\le 6N}$ is bounded from $L_2(\mathcal{M}_{\Gamma})$ to $\mathcal{M}_{\Gamma}$ is precisely a Haagerup-type inequality for $\mathcal{M}_{\Gamma}$, that is,  for every $d \ge 1$, the projection $P_d$ from $\mathcal{M}_{\Gamma}$ onto the homogeneous $L_2$-closed subspace generated by reduced operators of length $d$ extends to a bounded map from $L_2(\mathcal{M}_{\Gamma})$ to $\mathcal{M}_{\Gamma}$.
Moreover, this condition is equivalent to saying  that $\mathcal{M}_{\Gamma}$ is a graph product of finite-dimensional von Neumann algebras $\mathcal{M}_s$ satisfying
\begin{equation}\label{eq;M-finite-dimentional-condition}
 \sup_{s\in V\Gamma}\|\mathrm{id}:L^2(\mathcal M_s,\varphi_s)\to \mathcal M_s\|<\infty.
\end{equation}
 As an application of \ref{Thm A} and Proposition \ref{thm;equivalence-Pd-Haagerup-type}, we obtain the second main result of the paper which asserts the boundedness of our Hilbert transforms.

\begin{namedthm*}{Theorem B}\label{Thm B}
	Keep the assumptions of \ref{Thm A} and let $1<p<\infty$. Let $\mathcal{M}_{\Gamma}$ be the graph product of finite-dimensional von Neumann algebras satisfying (\ref{eq;M-finite-dimentional-condition}). 
    Then there exists a constant $c_p>0$ such that for every $x\in L_p(\mathcal{M}_{\Gamma})$,
    \begin{equation*}
        \|H_{\varepsilon}(x)\|_p\le c_p\|x\|_p\quad \text{and}\quad \|H_{\varepsilon}^{\mathrm{op}}(x)\|_p\le c_p\|x\|_p.
    \end{equation*}
\end{namedthm*}

Examples satisfying the assumptions of \ref{Thm B} include von Neumann algebras associated with graph products of finite groups (see Subsection~\ref{subsection;group vnA}), right-angled Hecke von Neumann algebras (see Subsection~\ref{subsection;Hecke}), and graph products of finite quantum groups (see Subsection~\ref{sec:quantum-groups}). 

Note that when the graph $\Gamma$ has no edges, that is, in the free product case, the compatibility condition \eqref{eq: compatibility} is automatically satisfied, and we no longer need to deal with $P_{\le 6N}(\mathcal{M}_{\Gamma})$. Consequently, the assumption in \ref{Thm B} that  von Neumann algebras are finite-dimensional and satisfy (\ref{eq;M-finite-dimentional-condition}) can be removed. In this setting, our result recovers the main theorem of \cite[Theorem 3.5]{mei2017free}. When the graph $\Gamma$ is finite, the compatibility assumption of the Hilbert transform in \ref{Thm B} can be removed; see Theorem \ref{thm;finite-graph-Lp-Hilbert-transform} for more details.
\medskip

For the Hilbert transform $H_{\varepsilon}^{\mathcal{L},d}$ depending on the
first $d$ blocks, its boundedness can be established by relating $\mathcal{L}_h$ to the operators $\mathcal{L}_s$ associated with the vertices of the graph.
 Let $\mathcal{M}_{\Gamma}$ denote either the group von Neumann algebra $\mathcal{L}G_\Gamma$ associated with a graph product of nontrivial finite groups, or the right-angled Hecke von Neumann algebra $\mathcal{M}_{\mathbf{q}}(W_\Gamma)$ associated with a multiparameter $\mathbf{q} = (q_s)_{s\in V\Gamma}$ in $\mathbb{R}^{V\Gamma}_{>0}$, where $\Gamma$ is a finite graph.

\begin{namedthm*}{Theorem C}\label{Thm C}
 Let $\mathcal{M}_{\Gamma}$ be as above and let $1<p<\infty$. Let $H_{\varepsilon}^{\mathcal{L},d}$ and $H_{\varepsilon}^{\mathcal{R},d}$ be the Hilbert transforms introduced in Definition \ref{def; d Hilbert}.
Then there exists a constant $C_p>0$ such that, for every
$x\in L_p(\mathcal{M}_{\Gamma})$,
\[
\|H_{\varepsilon}^{\mathcal{L},d}(x)\|_p
\leq
C_p \|x\|_p
\quad \text{and}\quad 
\|H_{\varepsilon}^{\mathcal{R},d}(x)\|_p
\leq
C_p \|x\|_p .
\]
\end{namedthm*}

As an application of \ref{Thm C}, we provide  a complete characterization of Ozawa's question for  graph products of finite groups and right-angled Hecke von Neumann algebras.
More precisely, these algebras satisfy Ozawa's compactness property if and only if every connected component of the underlying graph $\Gamma$ is a clique. We refer the reader to
Theorem~\ref{cor; Ozawa's problem} for further details.

\section{Preliminaries}

We begin by fixing some standard notation from graph theory and recalling the necessary background material. For the construction and basic properties of graph products of von Neumann algebras and noncommutative $L_p$-spaces, we refer the reader to \cite{caspers2017graph} and \cite{pisier2003non}, respectively.

\begin{defn}\label{defn:clique}
Let $\Gamma$ be a simplicial graph.
\begin{enumerate}
  \item[(i)] Let $s\in V\Gamma$. The \emph{link} of $s$ is defined by $\mathrm{Link}(s) := \{t\in V\Gamma: \{s,t\}\in E\Gamma\}$. 
  More generally, if $\Gamma'$ is a subgraph of $\Gamma$ and $s\in V\Gamma'$, then the link of $s$ in $\Gamma'$ is denoted by $\mathrm{Link}_{\Gamma'}(s)$.
   For $T \subseteq V\Gamma$, define $\mathrm{Link}(T) := \bigcap_{t\in T}\mathrm{Link}(t)$ with the convention that $\mathrm{Link}(\emptyset) = V\Gamma$. 
   \item[(ii)]   An \emph{induced subgraph} of a simplicial graph $\Gamma = (V\Gamma,E\Gamma)$ is a subgraph $\Gamma'=(V\Gamma',E\Gamma')$ 
that inherits all the edges $\{s,t\}$ in $E\Gamma$ whenever $s,t\in V\Gamma'$. In this case, we say that $\Gamma'$ is generated by $V\Gamma'$, and the corresponding graph product $G_{\Gamma'}$ canonically embeds as a subgroup of $G_\Gamma$.

\item [(iii)] A \emph{clique} in $\Gamma$ is an induced subgraph $\Gamma_0$ such that every pair of distinct vertices in $V\Gamma_0$ is connected by an edge.
For $n,k \geq 1$, we denote
 by $\mathrm{Cliq}(\Gamma,n)$ and $\mathrm{Cliq}(\Gamma,\ge k)$ the sets of cliques with $n$ vertices and at least $k$ vertices, respectively. We write $\mathrm{Cliq}(\Gamma)$ for the set of all cliques, and we include the empty set in $\mathrm{Cliq}(\Gamma)$.
\end{enumerate}

\end{defn}

\subsection*{Graph products of Hilbert spaces and von Neumann algebras}
In this subsection, we recall the construction of graph products of Hilbert spaces (Fock spaces) and graph products of von Neumann algebras (see \cite[Section 3]{caspers2017graph}).

Let $\Gamma = (V\Gamma, E\Gamma)$ be a simplicial graph. Let $W_{\Gamma}$ be the associated right-angled Coxeter group  and $\mathcal{W}_{\Gamma}$ the set of all reduced expressions of  elements in $W_{\Gamma}$.
If $\textbf{w} = s_1\ldots s_n\in \mathcal{W}_{\Gamma}$, then for any $\textbf{v} \simeq \textbf{w}$ there exists a unique permutation $\sigma\in S_n$  such that $\textbf{v} = s_{\sigma(1)}\ldots s_{\sigma(n)}$ and 
\begin{equation}\label{eq;tag-CO}
    \sigma(i)<\sigma(j)\; \text{whenever}\; 1\le i<j\le n\; \text{with}\;s_{\sigma(i)} = s_{\sigma(j)},
\end{equation}
see \cite[Lemma 2.3 (4)]{caspers2017graph}. 
For each nontrivial element $w\in W_{\Gamma}$, we fix a reduced expression and denote by $\mathcal{W}_{\Gamma}^{\text{f}}$ the collection of all such chosen representatives.
For each $s\in V\Gamma$, define 
$$\mathcal{W}_{\Gamma_s} := \{\textbf{w}\in \mathcal{W}_{\Gamma}^{\text{f}}: \ell(s\textbf{w})<\ell(\textbf{w})\}.$$ 
Equivalently, $\mathcal{W}_{\Gamma_s}$ consists of those reduced words in $\mathcal{W}_\Gamma^{\mathrm{f}}$ that admit a reduced expression beginning with $s$ (up to the braid relations).
We set  $$\mathcal{W}_{\Gamma_s}^c := \mathcal{W}_{\Gamma}^{\text{f}}\setminus \mathcal{W}_{\Gamma_s},$$
which consists of those reduced words in $ \mathcal{W}_{\Gamma}^{\text{f}}$ for which no equivalent reduced expression begins with $s$.

\medskip

	For each $s\in V\Gamma$, let $\mathcal{M}_s$ be a von Neumann algebra equipped with a normal faithful state $\varphi_s$. Let $(\mathcal{H}_s, {\rm id},\xi_s)$ be the GNS-construction associated with $\varphi_s$, so that we identify $\mathcal{M}_s\subseteq B(\mathcal{H}_s)$. Set $\mathcal{H}_s^\circ := (\mathbb{C}\xi_s)^\perp$. 
    For a word $\mathbf{w} = s_1\ldots s_n\in \mathcal{W}_{\Gamma}$, 
    define 
	$$\mathcal{H}_{\mathbf{w}}^{\circ} = \mathcal{H}_{s_1}^\circ\otimes\ldots\otimes \mathcal{H}_{s_n}^\circ.$$ 
    The \textit{graph product of Hilbert spaces} (or Fock space) $(\mathcal{H},\Omega)$ is defined by 
	$$\mathcal{H} := \mathbb{C}\Omega \oplus \underset{\mathbf{w}\in \mathcal{W}_{\Gamma}^{\text{f}}}{\bigoplus}\mathcal{H}_{\mathbf{w}}^{\circ} =  \mathbb{C}\Omega \oplus \underset{n\ge1}{\bigoplus}\underset{\underset{\ell(\mathbf{w}) = n}{\mathbf{w}\in  \mathcal{W}_{\Gamma}^{\text{f}}}}{\bigoplus}\mathcal{H}_{\mathbf{w}}^{\circ},$$
	where $\Omega$ is a unit vector.
For each $s\in V\Gamma$, set
$$\mathcal{H}(s):= \mathbb{C}\Omega \oplus \underset{\mathbf{w}\in \mathcal{W}_{\Gamma_s}^c}{\bigoplus}\mathcal{H}_{\mathbf{w}}^{\circ}.$$
Let $\mathbf{w} = s_1\ldots s_n\in \mathcal{W}_{\Gamma}$. Then there exists a unique $\mathbf{w}'\in \mathcal{W}_{\Gamma}^{\text{f}}$ such that  $\mathbf{w}
\simeq \mathbf{w}'$ via the braid relations. Moreover, there exists 
 a unique permutation $\sigma\in S_{n}$ satisfying (\ref{eq;tag-CO}) such that $\mathbf{w}' = s_{\sigma(1)}\ldots s_{\sigma(n)}$. We define a unitary operator $Q_{\mathbf{w}}: \mathcal{H}_{\mathbf{w}}^{\circ}\to \mathcal{H}_{\mathbf{w}'}^{\circ}$ by $$Q_{\mathbf{w}}(\xi_1\otimes\ldots\otimes\xi_n) = \xi_{\sigma(1)}\otimes\ldots\otimes\xi_{\sigma(n)}.$$

Fix $s\in V\Gamma$. We define a unitary operator $U_s: \mathcal{H}\to \mathcal{H}_s\otimes \mathcal{H}(s)$ via the following diagram

$$\begin{tikzpicture}[baseline=(current bounding box.center)]
	\matrix (m) [matrix of math nodes, row sep=3em, column sep=9em] {
		\mathcal{H} = (\mathbb{C}\Omega \oplus \mathcal{H}_s^\circ) \oplus \underset{\mathbf{w}\in \mathcal{W}_{\Gamma_s}^c}{\bigoplus}\mathcal{H}_{\mathbf{w}}^{\circ}\oplus \underset{\underset{\ell(\mathbf{w}) \ge 2}{\mathbf{w}\in  \mathcal{W}_{\Gamma_s}}}{\bigoplus}\mathcal{H}_{\mathbf{w}}^{\circ} \\
		 \mathcal{H}_s\otimes \mathcal{H}(s) = (\mathcal{H}_s\otimes\mathbb{C}\Omega)\oplus\underset{\mathbf{w}\in \mathcal{W}_{\Gamma_s}^c}{\bigoplus}(\mathbb{C}\xi_s\otimes \mathcal{H}_{\mathbf{w}}^{\circ})\oplus \underset{\mathbf{w}\in \mathcal{W}_{\Gamma_s}^c}{\bigoplus}(\mathcal{H}_s^\circ\otimes \mathcal{H}_{\mathbf{w}}^{\circ})\\
	};
	
	\draw[->] (m-1-1) ++(2.7,-0.2) -- ++(0,-1) node[midway, right] {};
	\draw[->] (m-1-1) ++(0.4,-0.2) -- ++(0,-1) node[midway, right] {};
	\draw[->] (m-1-1) ++(-1.5,-0.2) -- ++(0,-1) node[midway, right] {};
\end{tikzpicture}$$
by 
\begin{align*}
	U_s(\Omega) & = \xi_s\otimes \Omega;\\
	U_s(\xi) & = \xi\otimes \Omega,\quad \forall \xi\in \mathcal{H}_s^\circ;\\
	U_s(\xi) & = \xi_s\otimes\xi, \quad \forall \xi\in \mathcal{H}_{\mathbf{w}}^{\circ} \, \text{with } \mathbf{w}\in  \mathcal{W}_{\Gamma_s}^c,
\end{align*}
and for any $\xi\in \mathcal{H}_{\mathbf{w}}^{\circ}$ with $\mathbf{w}\in  \mathcal{W}_{\Gamma_s}$ and $\ell(\mathbf{w})\ge 2$,
$$U_s(\xi)  = Q_{s\mathbf{w}'}^*(\xi),$$ where $\mathbf{w}\simeq s\mathbf{w}'$ with $\mathbf{w}'\in \mathcal{W}_{\Gamma_s}^c$.


\begin{defn}
	For $s\in V\Gamma$, define a unital faithful and normal $*$-homomorphism $\lambda_s: B(\mathcal{H}_s)\to B(\mathcal{H})$ by
	$$\lambda_s(x) = U_s^*(x\otimes 1)U_s,\quad \forall x\in B(\mathcal{H}_s).$$
\end{defn}

We now recall the definition of graph products of von Neumann algebras, which generalize Voiculescu's free products of von Neumann algebras \cite{voiculescu2006symmetries} and tensor products of von Neumann algebras. 
\begin{defn}
	For each $s\in V\Gamma$, let $(\mathcal{M}_s, \varphi_s)$ be a von Neumann algebra equipped with a normal and faithful state $\varphi_s$. Let $(\mathcal{H}_s,\mathrm{id},\xi_s)$ be the corresponding GNS-construction and $\lambda_s$ be the $*$-homomorphism defined above. The \textit{graph product of von Neumann algebras} is defined  by
	$$\mathcal{M}_{\Gamma} := \big(\underset{s\in V\Gamma}{\bigcup}\lambda_s(\mathcal{M}_s)\big)''.$$
\end{defn}
In particular, if for each $s \in V\Gamma$, $\mathcal{M}_s = \mathcal{L}G_s$ is the group von Neumann algebra of a discrete group $G_s$, then
$
\mathcal{M}_\Gamma = \mathcal{L}G_\Gamma,
$
where $G_\Gamma$ is the graph product of the groups $\{G_s\}_{s \in V\Gamma}$. 

\begin{defn}
	An element $a=a_1\ldots a_n\in \mathcal{M}_{\Gamma}$ is called  \emph{a reduced operator of type} $s_1\ldots s_n$ if $a_i\in \mathcal{M}_{s_i}^\circ$ and $s_1\ldots s_n\in \mathcal{W}_{\Gamma}$, where $\mathcal{M}_{s_i}^\circ := \{b\in \mathcal{M}_{s_i}: \varphi_{s_i}(b) = 0\}$ for $i=1,\ldots,n$. The \emph{length} of such a reduced operator is defined as the length of the reduced expression $s_1\ldots s_n$, and we denote it by $\ell(a)$.
    Define the \emph{ set of initial
syllables} of $a$ by
\[
\operatorname{Init}(a)
= \{a_i: 1\le i\le n,\quad  s_j\in \mathrm{Link}(s_i) \text{ for every } j<i\}.\]
Equivalently, \(\operatorname{Init}(a)\) is the set of syllables that may occur as
the first syllable of some reduced expression for $a$. 
We define the \emph{maximal clique prefix} of $a$ by 
\[
\mathrm{MCP}(a)=\prod_{b\in \operatorname{Init}(a)} b.
\]
The factors in this product commute pairwise, so the product is well-defined.
\end{defn}



Note that for every $b\in \mathcal{M}_{s_i}$, one has the decomposition 
$$b= \varphi_{s_i}(b)\cdot \mathbf{1} +b^{\circ},  \quad \text{where } b^{\circ} = (b-\varphi_{s_i}(b)\cdot \mathbf{1})\in \mathcal{M}_{s_i}^\circ.$$ 
It follows that the linear span of $\mathbf{1}$ together with all reduced operators forms a dense $*$-subalgebra $\mathcal{F}_{\Gamma}$ of $\mathcal{M}_{\Gamma}$. Moreover, let $a = a_1\ldots a_n$ be a reduced operator  with $a_i\in \mathcal{M}_{s_i}^\circ$ and let $\mathbf{w}=s_1\ldots s_n\in \mathcal{W}_{\Gamma}^{\text{f}}$, then $$a\Omega = a_1\xi_{s_1}\otimes\ldots\otimes a_n\xi_{s_n}.$$
In particular,  $\Omega$ is a cyclic vector for $\mathcal{M}_{\Gamma}$. The normal state $\varphi_{\Gamma}$ on $\mathcal{M}_{\Gamma}$ associated with the GNS-construction $(\mathcal{H},\mathrm{id},\Omega)$ is called the \textit{graph product state}. It is the unique normal and faithful state on 
$\mathcal{M}_{\Gamma}$ such that $\varphi_{\Gamma}(a)=0$ for every reduced operator $a\in \mathcal{M}_{\Gamma}$. If $\mathbf{v}\in \mathcal{W}_{\Gamma}$ such that $\mathbf{v}\simeq \mathbf{w}$, then $$a=a_{\sigma(1)}\ldots a_{\sigma(n)},$$ where $\sigma\in  S_n$ is the unique permutation determined by \eqref{eq;tag-CO}
such that $\mathbf{v} = s_{\sigma(1)}\ldots s_{\sigma(n)}$.
\medskip

Fix $s\in V\Gamma$. It is straightforward to verify that for any $x\in \mathcal{M}_s$ and $\eta_1\otimes\ldots\otimes\eta_k\in \mathcal{H}_{\mathbf{w}}^{\circ}$ with $\mathbf{w}= s_1\ldots s_k\in \mathcal{W}_{\Gamma}^{\text{f}}$, one has
$$x(\eta_1\otimes\ldots\otimes\eta_k)= Q_{s\mathbf{w}}\big(x^\circ\xi_s\otimes\eta_1\otimes\ldots\otimes\eta_k\big) + \varphi_s(x)\eta_1\otimes\ldots\otimes\eta_k\quad \text{if } \; \mathbf{w}\in \mathcal{W}_{\Gamma_s}^c.$$
If $\mathbf{w}\in \mathcal{W}_{\Gamma_s}$, then by the braid relation, there exists a unique $\mathbf{w}'\in \mathcal{W}_{\Gamma_s}^c$ such that $\mathbf{w} \simeq s\mathbf{w}'$. Hence, there exists a unique permutation $\sigma\in S_{k}$  satisfying (\ref{eq;tag-CO}) such that $Q^*_{\mathbf{s\mathbf{w}'}}(\eta_1\otimes\ldots\otimes\eta_k)=\eta_{\sigma(1)}\otimes\ldots\otimes\eta_{\sigma(k)}\in \mathcal{H}_s\otimes \mathcal{H}(s)$. In this case, one checks that 
$$x(\eta_1\otimes\ldots\otimes\eta_k)= \eta_1\otimes\ldots\otimes(x\eta_{\sigma(1)}-\langle x\eta_{\sigma(1)},\xi_s\rangle\xi_s)\otimes\ldots\otimes\eta_k + \langle x\eta_{\sigma(1)},\xi_s\rangle\eta_{\sigma(2)}\otimes\ldots\otimes\eta_{\sigma(k)},$$
where the term $x\eta_{\sigma(1)}-\langle x\eta_{\sigma(1)},\xi_s\rangle\xi_s$ appears in the $\sigma(1)$-th  tensor component. 

For each $s\in V\Gamma$, let  $P_s$ denote the orthogonal projection from $\mathcal{H}$ onto $\underset{{\mathbf{w}\in  \mathcal{W}_{\Gamma_s}}}{\bigoplus}\mathcal{H}_{\mathbf{w}}^{\circ}$. It follows from the above discussion that, for $x\in \mathcal{M}_s^\circ$, the operators 
$$P_s xP_s^\perp, \quad P_s xP_s, \quad P_s^\perp xP_s$$ 
correspond, respectively, to a (left) creation operator, a diagonal operator and an annihilation operator. Moreover, these three classes of operators generate the reduced operators in $\mathcal{M}_\Gamma$ (see \cite[Lemma 2.2]{BORST2024110350} or \cite[Lemma 2.5]{caspers2021graph}).

\subsection*{Noncommutative $L_p$-spaces}

Let $(\mathcal{M},\tau)$ be a finite von Neumann algebra equipped with a normal faithful tracial state $\tau$. For $1\le p<\infty$, let $L_p(\mathcal{M})$ denote the completion of $\mathcal{M}$ with respect to the norm
\[
\|x\|_p
:=
\tau(|x|^p)^{1/p},
\qquad x\in\mathcal{M}.
\]
We also set
$L_\infty(\mathcal{M}) := \mathcal{M}$.
We recall below several standard facts about noncommutative $L_p$-spaces that will be used throughout the paper; we refer the reader to \cite{dodds1990remarks,Pi1998, pisier2003non} for further details.

Since $\mathcal{M}$ is finite, one has the embeddings
$L_{p_1}(\mathcal{M})\subseteq L_{p_0}(\mathcal{M})$ contractively for any $1\le p_0\le p_1\le\infty$.
Thus, $(L_{p_0}(\mathcal{M}),L_{p_1}(\mathcal{M}))$ forms an interpolation pair. Their complex interpolation space is given by
\[
(L_{p_0}(\mathcal{M}),L_{p_1}(\mathcal{M}))_\theta
=
L_p(\mathcal{M}),
\]
where $\frac1p
=
\frac{1-\theta}{p_0}
+
\frac{\theta}{p_1}$ and $0<\theta<1$.
In particular, one has the noncommutative version of the Riesz-Thorin interpolation theorem. Let $1\le p_i,q_i\le\infty$ for $i=0,1$, and let
$\frac1p
=
\frac{1-\theta}{p_0}
+
\frac{\theta}{p_1}$ and $\frac1q
=
\frac{1-\theta}{q_0}
+
\frac{\theta}{q_1}$.
If $T:L_{p_i}(\mathcal{M})\to L_{q_i}(\mathcal{M})$
is bounded with operator norm denoted by $\|T\|_{p_i,q_i}$ for $i=0,1$, then
$T:L_p(\mathcal{M})\to L_q(\mathcal{M})$
is bounded and satisfies
\[
\|T\|_{p,q}
\le
\|T\|_{p_0,q_0}^{\,1-\theta}
\|T\|_{p_1,q_1}^{\,\theta}.
\]

Let $1 \le p < \infty$ and let $q$ be such that $\frac{1}{p}+\frac{1}{q}=1
$. Then the dual space $L_p(\mathcal{M})^*$ can be identified isometrically with $ L_{q}(\mathcal{M})$  via the duality $\langle x, y\rangle_{L_p, L_q}=\tau (xy)$. More generally, H\"older's inequality remains valid in the noncommutative setting. Let $1 \le p,q,r \le \infty$ satisfy $\frac{1}{r} = \frac{1}{p}+\frac{1}{q}$. Then for any $x\in L_p(\mathcal{M})$ and $y\in L_q(\mathcal{M})$, one has 
        $$xy\in L_r(\mathcal{M})\quad \text{and}\quad \|xy\|_r\le \|x\|_p\|y\|_q.$$

\section{Cotlar identities on graph products of finite von Neumann algebras}\label{sec; Section 2}

Let $\Gamma$ be a simplicial graph. For the rest of the paper, we restrict ourselves to graph products of finite von Neumann algebras equipped with normal faithful tracial states, for simplicity. However, the arguments should extend, with suitable modifications, to the setting of type~III von Neumann algebras equipped with faithful normal states. For each $s \in V\Gamma$, let $(\mathcal{M}_s,\varphi_s)$ be such a pair, and denote by $\mathcal{M}_\Gamma$ their graph product.
Then $\mathcal{M}_\Gamma$ is itself a finite von Neumann algebra. Indeed, by \cite[Theorem 1]{mlotkowski2004lambda}, the graph product state $\varphi_\Gamma$ associated with the family $\{(\mathcal{M}_s,\varphi_s)\}_{s\in V\Gamma}$ is tracial.

\medskip 

As we mentioned in the introduction, for graph products of finite von Neumann algebras, the Cotlar identity holds only outside a subset and for which conditional expectations do not exist. In \cite[Subsection 4.1]{mei2017free}, Mei and Ricard developed a strategy to handle Hilbert transforms depending on the $d$th syllable, and in \cite[Theorem 2.4]{mei2024lambda}, Mei applied a similar approach to Hilbert transforms on hyperbolic groups. This strategy is essentially based on two ingredients: a generalized Cotlar identity and a Haagerup-type inequality. In the following, we formalize this approach for general $L_2$-bounded linear operators on graph products of finite von Neumann algebras.

\begin{defn}
    For $d\geq 1$, let $P_d$ denote the projection from $L_2(\mathcal{M}_{\Gamma})$ onto the  homogeneous $L_2$-closed subspace $\mathcal{M}_d$ generated by all reduced operators of length $d$. Let $P_{\le d}$ denote the projection from $L_2(\mathcal{M}_{\Gamma})$ onto the subspace $$\mathcal{M}_{\leq d}:= \mathbb{C}\mathbf{1}\oplus\bigoplus_{k=1}^{d}\mathcal{M}_k.$$ Set $P_{>d} := {\rm id}-P_{\le d}$.
\end{defn}

\begin{defn}
    Let $d\geq 1$ and $T: L_2(\mathcal{M}_{\Gamma})\to L_2(\mathcal{M}_{\Gamma})$ be a bounded linear operator. We say that $T$ satisfies the generalized Cotlar identity with respect to $P_{>d}$ if 
    \begin{equation}\label{eq;Cotlar-Pd}\tag{$\mathrm{Cotlar}_{P_{>d}}$}
		P_{> d}\big(T(x)T(y)^*\big) = P_{> d}\Big(	T\big(x	T(y)^*\big) + 	T\big(	yT(x)^*\big)^*- 	T\big(	T(xy^*)^*\big)^*\Big),
     \end{equation}   
for any $x,y\in \mathcal{F}_{\Gamma}$.
\end{defn}

If $T$ satisfies (\ref{eq;Cotlar-Pd}),  the remaining difficulty lies in handling the restriction of $T$ to $P_{\le d}(\mathcal{M}_{\Gamma})$. The following theorem shows that if $P_{\le d}$ is bounded from $L_2(\mathcal{M}_{\Gamma})$ to $\mathcal{M}_{\Gamma}$, then $T$ is bounded on $L_p(\mathcal{M}_{\Gamma})$ for all $2\le p<\infty$. 


\begin{thm}\label{thm;Hilbert-trans-finite-vnA}
    Let $d\geq 1$, $2\le p<\infty$, $x\in L_p(\mathcal{M}_{\Gamma})$ and $T:L_2(\mathcal{M}_{\Gamma})\to L_2(\mathcal{M}_{\Gamma})$ be a bounded operator. 
    Suppose that $P_{\le d}$ extends to a  bounded map from $L_2(\mathcal{M}_{\Gamma})$ to $\mathcal{M}_{\Gamma}$ and $T$ satisfies \eqref{eq;Cotlar-Pd}.
    Then there is a constant $C_p>0$ such that
    \begin{equation*}
        \|T(x)\|_p\le C_p\|x\|_p.\quad 
    \end{equation*}
\end{thm}
\begin{proof}
      We first show the $L_{2^k}$-boundedness of $T$ by induction on $k\geq 1$.
	  Assume that $T$ is bounded on $L_{2^k}(\mathcal{M}_{\Gamma})$ and set $C_{2^k} := \|T: L_{2^k}(\mathcal{M}_{\Gamma})\to L_{2^k}(\mathcal{M}_{\Gamma})\|$. Now we show that $T$ is bounded on $L_{2p}$, where $p= 2^k$.
	 Note that 
	 $$|T(x)^*|^2 = P_{\le d}|T(x)^*|^2 + P_{> d}\big(T(x)T(x)^*\big).$$
      Since $P_{\le d}$ is bounded from $L_2(\mathcal{M}_{\Gamma})$ to $\mathcal{M}_{\Gamma}$, duality implies that $P_{\le d}$ is bounded from $L_1(\mathcal{M}_{\Gamma})$ to $L_2(\mathcal{M}_{\Gamma})$. 
      H\"older's inequality and the $L_2 \rightarrow L_\infty$  boundedness yield that $P_{\le d}$ is  bounded from $L_2(\mathcal{M}_{\Gamma})$ to $L_q(\mathcal{M}_{\Gamma})$ for all $q\ge1$. 
      Combining this with the $L_1 \rightarrow L_2$ boundedness, we conclude that $P_{\le d}$ is  bounded from $L_1(\mathcal{M}_{\Gamma})$ to $L_q(\mathcal{M}_{\Gamma})$ for all $q\ge 1$. Together with the assumption that  $T$ satisfies (\ref{eq;Cotlar-Pd}), we get
	 \begin{equation*}
	 	\begin{aligned}
	 		\|T(x)\|_{2p}^2 & = \||T(x)^*|^2\|_p = \|P_{\le d}|T(x)^*|^2 + P_{> d}\big(T(x)T(x)^*\big)\|_p\\
	 		& \le \|P_{\le d}|T(x)^*|^2\|_p + \|P_{> d}\big(T(x)T(x)^*\big)\|_p\\
	 		& \le C\|T(x)T(x)^*\|_1+ C' \big(\|T\big(xT(x)^*\big)\|_p+ \|T\big(xT(x)^*\big)\|_p + \|T\big(	T(xx^*)^*\big)\|_p\big)\\
	 		& \le C\|T(x)\|_2^2 + 2C'C_p\|xT(x)^*\|_p  + C'C_p^2\|xx^*\|_p\\
	 		& \le CC_2^2\|x\|_{2p}^2 + 2C'C_p\|x\|_{2p}\|T(x)\|_{2p} +C'C_p^2\|x\|_{2p}^2, 
	 	\end{aligned}
	 \end{equation*}
     where $C:= \|P_{\le d}: L_1(\mathcal{M}_{\Gamma})\to L_p(\mathcal{M}_{\Gamma})\|$ and $C':= \|P_{>d}: L_{p}(\mathcal{M}_{\Gamma})\to L_p(\mathcal{M}_{\Gamma})\|$.
	 The last inequality follows from H$\ddot{\mathrm{o}}$lder's inequality and the fact that the trace on $\mathcal{M}_{\Gamma}$ is finite. Consequently, 
     $$\|T(x)\|_{2p}\le (1+\sqrt{2})(\sqrt{C}C_2+C'C_p)\|x\|_{2p},$$
	 which yields the $L_{2^k}$-boundedness of $T$ for all $k\geq 1$.
	 The $L_p$-boundedness of $T$ for all $2\le p<\infty$ then follows by interpolation.
\end{proof}


\subsection{Hilbert transforms}
In graph products of von Neumann algebras, due to commutativity,  the initial syllable of reduced operators is not necessarily unique.
This gives rise to two natural families of projections: the first focuses on the initial syllable, while the second focuses on the maximal clique prefix. In what follows, we introduce these projections and explain how they are used to define a Hilbert transform as an unconditional sum of these projections.

In the next subsection, we prove that the Hilbert transform defined here satisfies the generalized Cotlar identity \eqref{eq;Cotlar-Pd}, for some $d$ depending only on the graph.


\begin{defn}\label{def;Ls}
	Let $s\in V\Gamma$  and $\Gamma_0\in \mathrm{Cliq}(\Gamma, \ge2)$. 
	\begin{enumerate}
		\item[(i)] Let $L_s$ denote the projection from $\mathcal{F}_{\Gamma}$ onto the subspace
        $\mathcal{M}_{\Gamma_s}$  generated by  all reduced operators  of type $\mathbf{w}$ satisfying $\ell(s\mathbf{w})< \ell(\mathbf{w})$. Equivalently, this condition means that there exists   $\mathbf{v}\in \mathcal{W}_\Gamma$ with $\mathbf{v}\simeq \mathbf{w}$ such that $\mathbf{v}$ begins with $s$. 
        We write $L_s^\perp:={\rm id}-L_s$.
	\item[(ii)] Let  $\mathcal{L}_{s}$ denote the projection from $\mathcal{F}_{\Gamma}$ onto the subspace $\mathcal{M}_{\Gamma,s}$ generated by  all reduced operators of type $\mathbf{w}$ such that for every  $\mathbf{v}\in \mathcal{W}_{\Gamma}$ with $\mathbf{v}\simeq\mathbf{w}$, the reduced word $\mathbf{v}$ begins with $s$.  
    Let $\mathcal{R}_s$ denote the projection from $\mathcal{F}_{\Gamma}$ onto the subspace $\mathcal{M}_{\Gamma,s}^*$. Note that every reduced operator $a$ in $\mathcal{M}_{\Gamma,s}$ satisfies $\ell (\mathrm{MCP}(a))=1$.
 \item [(iii)]
 Let $\mathcal{L}_{\Gamma_0}$ denote the projection from $\mathcal{F}_{\Gamma}$ onto  the subspace $\mathcal{S}_{\Gamma_0}$ generated by all reduced operators $a$ such that
$\mathrm{MCP}(a) = a_{i_1}\cdots a_{i_k}$, where $a_{i_j} \in \mathcal{M}_{s_{i_j}}^\circ$
and $\{s_{i_1},\ldots,s_{i_k}\} = V\Gamma_0$. 
Let $\mathcal{R}_{\Gamma_0}$ denote the projection from $\mathcal{F}_{\Gamma}$ onto the subspace $\mathcal{S}_{\Gamma_0}^*$.
\end{enumerate}
	  
\end{defn}

Note that the projections $L_s$, $s \in V\Gamma$, are not pairwise orthogonal in general, whereas the projections $\mathcal{L}_s$ and $\mathcal{L}_{\Gamma_0}$ are mutually orthogonal families. In particular, the operator
$$\sum_{s\in V\Gamma} L_s + \varphi_{\Gamma}\cdot \mathbf{1}$$
is not a projection. However, each individual projection $L_s$ extends to a bounded operator on $L_p(\mathcal{M}_\Gamma)$ for $1 < p < \infty$. To see this, we use the decomposition of graph products as amalgamated free products.
By \cite[Theorem 1.1]{caspers2017graph}, for each $s \in V\Gamma$, one has
 $$(\mathcal{M}_s\bar\otimes \mathcal{M}_{\mathrm{Link}(s)})\ast_{\mathcal{M}_{\mathrm{Link}(s)}} \mathcal{M}_{\Gamma\setminus \{s\}},$$
where $\mathcal{M}_{\mathrm{Link}(s)}$ and $\mathcal{M}_{\Gamma\setminus \{s\}}$ denote the graph products associated with induced subgraphs generated by $\mathrm{Link}(s)$ and $V\Gamma\setminus \{s\}$, respectively. 
Applying the Hilbert transform $H_{\varepsilon}$ of \cite{mei2017free} to this amalgamated free product decomposition with amalgamation over $\mathcal{M}_{\mathrm{Link}(s)}$,  we obtain 
$$H_{\varepsilon} = \varepsilon_s L_s + \varepsilon_{s}' (L_s^\perp - \mathbb{E}_s) + \varepsilon_0 \mathbb{E}_s,$$
where $\varepsilon = (\varepsilon_s,\varepsilon'_{s},\varepsilon_0) \in \{\pm1\}^3$, the operators $L_s$ and $L_s^\perp$ are the projections in Definition \ref{def;Ls}(i), and $\mathbb{E}_s: \mathcal{M}_{\Gamma} \to \mathcal{M}_{\mathrm{Link}(s)}$ denotes the canonical conditional expectation. Consequently, the $L_p$-boundedness of $L_s$ on $L_p(\mathcal{M}_\Gamma)$ for $1 < p < \infty$ follows from that of $H_{\varepsilon}$ (see \cite[Corollary 3.9]{mei2017free}).

Now we turn to the projections $\mathcal{L}_s$ and $\mathcal{L}_{\Gamma_0}$ introduced in Definition \ref{def;Ls}(ii) and (iii).

\begin{defn}\label{defn;Hilbert transform}
	Define the Hilbert transforms $H_{\varepsilon}$ and $H_{\varepsilon}^\mathrm{op}$ on $\mathcal{F}_{\Gamma}$ by 
	\begin{equation*}
		H_{\varepsilon} := \sum_{s\in V\Gamma}\varepsilon_s\mathcal{L}_s +\sum_{\Gamma_0\in \mathrm{Cliq}(\Gamma,\ge2)} \varepsilon_{\Gamma_0} \mathcal{L}_{\Gamma_0} + \varepsilon_0\varphi_{\Gamma}\quad \text{and}\quad H_{\varepsilon}^\mathrm{op} := \sum_{s\in V\Gamma}\varepsilon_s\mathcal{R}_s + \sum_{\Gamma_0\in \mathrm{Cliq}(\Gamma,\ge2)} \varepsilon_{\Gamma_0} \mathcal{R}_{\Gamma_0} + \varepsilon_0\varphi_{\Gamma},
	\end{equation*}
	where $\varepsilon_0,\varepsilon_{\Gamma_0},\varepsilon_s\in \{\pm 1\}$ for $s\in V\Gamma$, and $\varepsilon = (\varepsilon_0,\varepsilon_{\Gamma_0},\varepsilon_s) \in \{\pm1\}^{\#\mathrm{Cliq}(\Gamma)}$.
    \end{defn}

\begin{rem}
  It is straightforward to verify that
\[
H_{\varepsilon} H_{\varepsilon}^{\mathrm{op}} = H_{\varepsilon}^{\mathrm{op}} H_{\varepsilon},
\qquad
H_{\varepsilon}^{\mathrm{op}}(a) = H_{\varepsilon}(a^*)^*, \quad a \in \mathcal{M}_\Gamma.
\]
If the underlying graph $\Gamma$ has no edges, then the operators $H_{\varepsilon}$ and $H_{\varepsilon}^{\mathrm{op}}$ defined above reduce to the corresponding Hilbert transforms introduced in \cite{mei2017free}.
\end{rem}

\begin{rem}
	Note that there is another natural way to decompose graph products.  
    Suppose that $V\Gamma=\{s_i:1\leq i\leq M\}$, where $M\in \mathbb{N}$ or $M=\infty$. For each $2\leq k\leq M$, let $\mathfrak{L}_{s_k}$
	denote the projection from $\mathcal{F }_\Gamma$ onto the subspace generated by all reduced operators of type $\mathbf{w}$ satisfying
	$\ell(s_k\mathbf{w})<\ell(\mathbf{w})$, and $
	\ell(s_j\mathbf{w})>\ell(\mathbf{w})$
	for every $1\le j<k$.  Equivalently, this is the subspace generated by all reduced operators 
    that admit a reduced expression starting with $\mathcal{M}_{s_k}^\circ$, but do not admit a reduced expression starting with $\mathcal{M}_{s_j}^\circ$ for any $1\le j<k$. Therefore, we have 
    $$ \mathfrak{L}_{s_k} ={L}_{s_k} \prod_{\substack{1\leq j< k}} L_{s_j}^\perp$$ for any $ 2\leq k\leq M$, and they are pairwise orthogonal. 
	We define the corresponding Hilbert transform 
    $	\widetilde{H}_{\varepsilon}$ by
    \begin{equation*}
        \widetilde{H}_{\varepsilon} =
	\varepsilon_{s_1}L_{s_1} +
	\sum_{2\leq k\leq M}\varepsilon_{s_k}\mathfrak{L}_{s_k} +
	\varepsilon_0\varphi_\Gamma,
    \end{equation*}
	where $\varepsilon_0,\varepsilon_{s_1}, \varepsilon_{s_k}\in\{\pm1\}$ for all $2\leq k\leq M$, and $\varepsilon = (\varepsilon_0,\varepsilon_{s_i})_{1\leq i\leq M}$.

	It turns out that $\widetilde{H}_{\varepsilon}$ is a special case of the Hilbert transform $H_\varepsilon$ introduced in Definition~\ref{defn;Hilbert transform}. Let $N_{s_1}$ be the set of all cliques in $\mathrm{Cliq}(\Gamma,\ge 2)$ whose vertices include $s_1$.
	For each $2\leq k\leq M$, let
	$$N_{s_k} = \left\{
	\Gamma_0\in\mathrm{Cliq}(\Gamma,\ge2): s_k\in V\Gamma_0,\
	s_j\notin V\Gamma_0\ \text{for every }1\le j<k
	\right\}.$$
	Then $\{N_{s_i}\}_{1\leq i\le M}$ forms a partition of $\mathrm{Cliq}(\Gamma,\ge2)$.
	Choose the signs so that
	$	\varepsilon_{\Gamma_0} = \varepsilon_{s_k}$, for any $\Gamma_0\in N_{s_k}$.
	Then we obtain
	$$\varepsilon_{s_1}L_{s_1} = \varepsilon_{s_1}\mathcal{L}_{s_1}+ \varepsilon_{s_1}\sum_{\Gamma_0\in N_{s_1}}\mathcal{L}_{\Gamma_0}\quad \text{and}\quad
	\varepsilon_{s_k}\mathfrak{L}_{s_k} = \varepsilon_{s_k}\mathcal{L}_{s_k}+
	\sum_{\Gamma_0\in N_{s_k}}
	\varepsilon_{\Gamma_0}\mathcal{L}_{\Gamma_0}.$$
	Summing over all $1\leq i\leq M$ yields $\widetilde{H}_{\varepsilon} = H_{\varepsilon}.$
\end{rem}

Now we show that the projections $\mathcal{L}_s$  and $\mathcal{L}_{\Gamma_0}$ are completely bounded on $L_p(\mathcal{M}_{\Gamma})$.
Moreover, when $\Gamma$ is a finite simplicial graph, this implies that the Hilbert transforms $H_{\varepsilon}$ and $H_{\varepsilon}^{\mathrm{op}}$ in Definition \ref{defn;Hilbert transform} are also completely bounded on $L_p(\mathcal{M}_{\Gamma})$. 

\begin{thm}\label{thm;finite-graph-Lp-Hilbert-transform}
Let $\Gamma$ be a simplicial graph satisfying \eqref{eq; Fi1} and $1<p<\infty$.
\begin{itemize}
    \item[(i)] For every $s\in V\Gamma$ and $\Gamma_0\in\mathrm{Cliq}(\Gamma,\ge 2)$, 
the projections $\mathcal{L}_s$ (resp. $\mathcal{R}_s$) and $\mathcal{L}_{\Gamma_0}$ (resp. $\mathcal{R}_{\Gamma_0}$) are completely bounded on $L_p(\mathcal{M}_{\Gamma})$.

    \item[(ii)] If $\Gamma$ is finite, then there is a constant $c_{p,\Gamma}>0$ depending on $p, |V\Gamma|$ and $\#\mathrm{Cliq}(\Gamma)$ (the number of cliques in $\Gamma$) such that, for every $n\ge1$ and $x\in  L_p(\mathbb M_n\bar{\otimes}\mathcal{M}_{\Gamma})$,  
    $$\|({\rm id}_{\mathbb M_n}\otimes H_{\varepsilon})(x)\|_p\le c_{p,\Gamma}\|x\|_p\quad \text{and}\quad \|({\rm id}_{\mathbb M_n}\otimes H_{\varepsilon}^{\mathrm{op}})(x)\|_p\le c_{p,\Gamma}\|x\|_p.$$
\end{itemize}
\end{thm}
\begin{proof}
(i) It suffices to prove the $L_p$-boundedness of $\mathcal{L}_s$ and $\mathcal{L}_{\Gamma_0}$. 
For every $s\in V\Gamma$, by the graph product decomposition in \cite[Theorem 1.1]{caspers2017graph}, we have 
    \[ \mathcal{M}_\Gamma = \bigl( \mathcal{M}_{s}\,\bar\otimes\, \mathcal{M}_{\mathrm{Link}(s)} \bigr) *_{\mathcal{M}_{\mathrm{Link}(s)}} \mathcal{M}_{\Gamma\setminus\{s\}}. \]
Hence, the complete $L_p$-boundedness of Hilbert transforms on amalgamated free products of finite von Neumann algebras \cite[Corollary 3.9]{mei2017free} implies that  there is a constant $C_p>0$, independent of $s$, such that 
 $$\|L_{s}\|_{\rm cb,p\to p}\le C_p.$$
     For each $s\in V\Gamma$, the projection $\mathcal{L}_{s}$ can be written as
     $$\mathcal{L}_{s} = L_{s}\prod_{\substack{t\in \mathrm{Link}(s)}} L_{t}^{\perp}.$$
    More generally, if $\Gamma_0\in\mathrm{Cliq}(\Gamma,\ge 2)$, then 
    \[ \mathcal{L}_{\Gamma_0} = \left\{\begin{array}{lr}
    \displaystyle
			\prod_{s\in V\Gamma_0}L_s, \quad \text{if } \; \mathrm{Link}(\Gamma_0) = \emptyset,\\[4ex]
            \displaystyle
			\prod_{s\in V\Gamma_0}L_s \prod_{t\in\mathrm{Link}(\Gamma_0)}L_t^\perp, \quad \text{if } \; \mathrm{Link}(\Gamma_0) \ne \emptyset.
		\end{array} \right.
        \]
    Therefore, the complete $L_p$-boundedness of $L_{s}$ ($s\in V\Gamma$) implies that of $\mathcal{L}_{s}$ and $\mathcal{L}_{\Gamma_0}$, with bounds controlled by $(1+C_p)^{N+1}$ and $C_p^{|V\Gamma_0|}(1+C_p)^{N+1-|V\Gamma_0|}$, respectively.

    (ii) 
     It follows immediately from (i) and the definitions of $H_{\varepsilon}$ and $H_{\varepsilon}^{\mathrm{op}}$ that  
    \[ \|({\rm id}_{\mathbb M_n}\otimes H_{\varepsilon})(x)\|_p \leq \#\mathrm{Cliq}(\Gamma)(1+C_p)^{|V\Gamma|}\big\|x\big\|_p, \]
   and 
    \[ \|({\rm id}_{\mathbb M_n}\otimes H_{\varepsilon}^{\mathrm{op}})(x)\|_p \leq \#\mathrm{Cliq}(\Gamma)(1+C_p)^{|V\Gamma|}\big\|x\big\|_p. \]
    Thus, the asserted estimates hold with $c_{p,\Gamma}: =\#\mathrm{Cliq}(\Gamma)\cdot (1+C_p)^{|V\Gamma|}$.
\end{proof}

We observe that the bounds of Hilbert transforms obtained above tend to infinity when $\Gamma$ is an infinite simplicial graph. In this case, we will need to explore a suitable  Cotlar identity on the graph products. 
The validity of the Cotlar identity in \cite{mei2017free} crucially depends on the uniqueness of reduced expressions in free products of von Neumann algebras, which fails in the context of graph products. As a consequence, the Cotlar identity is not in general valid for the Hilbert transforms introduced in Definition~\ref{defn;Hilbert transform}.
We illustrate this failure with a simple example. Let $a=a_1a_2$ and $b^*=b_1b_2$ be reduced operators in $\mathcal{M}_\Gamma$ of types $s_1s_2$ and $t_1t_2$, respectively. Assume that the induced subgraph generated by $\{s_1,s_2,t_1,t_2\}$ is 

\vspace{1pt}
\begin{center}
	\begin{tikzpicture}
		\coordinate[label = below: $s_1$] (s) at (0,0);
		\coordinate[label = below: $t_1$] (t) at (1,0);
		\coordinate[label = below: $s_2$] (r) at (2,0);
		\coordinate[label = below: $t_2$] (k) at (3,0); 
		
		\node at (s)[circle, fill, inner sep=1pt]{};
		\node at (t)[circle, fill, inner sep=1pt]{};
		\node at (r)[circle, fill, inner sep=1pt]{};
		\node at (k)[circle, fill, inner sep=1pt]{};

		\draw (s)--(t) node[above, midway]{};
		\draw (t)--(r) node[above, midway]{};
		\draw (r)--(k) node[above, midway]{};
	\end{tikzpicture}
    \captionof{figure}{The induced subgraph generated by $\{s_1,s_2,t_1,t_2$\}.}
    \label{graph;1}
	\end{center}
    \medskip
In particular, this configuration implies that $s_1s_2 \in \mathrm{Link}(t_1)$ and $t_1t_2 \in \mathrm{Link}(s_2)$, while $s_1s_2 \neq s_2s_1$ and $t_1t_2 \neq t_2t_1$.
In this situation, the Cotlar identity from \cite{mei2017free} fails, i.e.
$$H_{\varepsilon}(a)H_{\varepsilon'}^{\mathrm{op}}(b^*) \neq H_{\varepsilon}\big(a	H_{\varepsilon'}^{\mathrm{op}}(b^*)\big) + 	H_{\varepsilon'}^{\mathrm{op}}\big(	H_{\varepsilon}(a)b^*\big)- 	H_{\varepsilon'}^{\mathrm{op}}	H_{\varepsilon}(ab^*),$$
since, at the level of coefficients of $ab^*$, one obtains
$$\varepsilon_{s_1}\varepsilon_{t_2}'\ne \varepsilon_{t_2}'\varepsilon_{\Gamma_1}+ \varepsilon_{s_1}\varepsilon_{\Gamma_2}' - \varepsilon_{\Gamma_1}\varepsilon_{\Gamma_2}',$$
where $\Gamma_i$ is the clique generated by $\{s_i,t_i\}$ for $i= 1,2$. 

Although $H_{\varepsilon}$ and $H_{\varepsilon}^{\mathrm{op}}$ fail to satisfy the Cotlar identity on the entire von Neumann algebra, a generalized Cotlar identity established in the next subsection shows that when we assume some conditions on the coefficients of  $H_{\varepsilon}$,
it remains valid for operators of sufficiently large length, with the threshold depending only on the graph $\Gamma$.

\subsection{Generalized Cotlar identities} 
In \cite[Lemma 3.2]{Laura2013}, Ciobanu, Holt, and Rees gave a useful characterization of block lengths under the group operation, expressed via factorizations in graph products of groups, which we state below.

\begin{lem}(\cite[Lemma 3.2]{Laura2013})\label{lem;reduced-decompsition-groups}
    Let $G_{\Gamma}$ be a graph product of groups associated with $\Gamma$. Let $k,n \geq 1$ and $g,h\in G_{\Gamma}\setminus \{e\}$ with $\ell(g) = k$ and $\ell(h) = n$. Suppose that $\ell(gh) = k+n-q$ with $q \geq 0$. Then there exist factorizations of $g$ and $h$
    \begin{equation*}
        g = g_1g_2'g_3 \quad \text{and}\quad h= g_3^{-1}g_2''h_1
    \end{equation*}
    such that 
    \begin{enumerate}
       \item[(i)] $\ell(g) = \ell(g_1) + \ell(g_2') +\ell(g_3)$ and $\ell(h) = \ell(g_3) + \ell(g_2'') +\ell(h_1)$,
    
        \item [(ii)] there exist $r\geq 0$ and $\Gamma_0\in \mathrm{Cliq}(\Gamma,r)$ such that $g_2',g_2''\in G_{\Gamma_0}$ and 
        $\ell(g_2') = \ell(g_2'') = \ell(g_2'g_2'') = r$,

        \item[(iii)] $q = r+ 2\ell(g_3)$.
    \end{enumerate}
\end{lem}

We now show that the above lemma admits an analogue in the context of graph products of von Neumann algebras. 
\begin{prop}\label{prop; ab-decomposition}
Let $a$ and $b$ be two reduced operators in $\mathcal{M}_{\Gamma}$ such that $\ell (a)=k$ and $\ell(b)=n$. 
   Then there exist factorizations of $a$ and $b$ given by
    	$$a = a'ca{''}\quad \text{and}\quad b = a{''^{-1}}c{'}b'$$
    such that 
    	\begin{enumerate}
     \item[(i)] $\ell(a) = \ell(a') + \ell(c) +\ell(a'')$ and $\ell(b) = \ell(a'') + \ell(c') +\ell(b')$,

        \item[(ii)] there exist $r\geq 0$ and $\Gamma_0\in \mathrm{Cliq}(\Gamma,r)$ such that 
        $c = c_{1}\ldots c_{r}$ and $c' = c_{1}'\ldots c_{r}'$ are reduced operators of the same type $s'_1\ldots s'_r\in \mathcal{W}_{\Gamma_0}$, 
        
			
		\item[(iii)] if \((c_ic_i')^\circ\ne0\) for every \(i=1,\ldots,r\), then $\ell\big((c_{1}c_{1}'\big)^\circ\ldots (c_{r}c_{r}')^\circ) = r$ and $\ell\big(a'(c_{1}c_{1}'\big)^\circ\ldots (c_{r}c_{r}')^\circ b') = k+n-q$ with $q = r +2 \ell(a{''})$.
	\end{enumerate}
\end{prop}

\begin{proof}
Suppose that $a = a_1\ldots a_k, b = b_1\ldots b_n\in \mathcal{M}_{\Gamma}$ are  reduced operators of type $s_1\ldots s_k$ and type $t_1\ldots t_n$, respectively. 
Let $w = s_1\ldots s_k$ and $v = t_1\ldots t_n$ be the corresponding elements in the right-angled Coxeter group $W_{\Gamma}$. 
We first treat the case where no cancellation occurs in the product of  $a$ and $b$, that is 
when we reduce the expression $ab$, there is no  product of syllables which becomes identity. In this case we have $a{''} = \mathbf{1}$. 
Note that there are factorizations  $$w = \mathbf{w}_1\mathbf{w}_2, \quad  v = \mathbf{v}_1\mathbf{v}_2$$ 
such that $\mathbf{w}_2 = \mathbf{v}_1^{-1}$ and the generators of $\mathbf{w}_{2}$ commute with each other. Such  factorizations may not be unique due to the braid relation. We choose the one for which  $\ell(\mathbf{w}_2)$ is maximal and we denote $r:=\ell(\mathbf{w}_2)$.  If  $r= 0$, then $c=c' = \mathbf{1}$, and the result is immediate. If $r\ge 1$,  we define $c$ and $c'$ as the reduced operators corresponding to $\mathbf{w}_{2}$ and  $\mathbf{v}_{1}^{-1}$, respectively, and $a'$, $b'$ as those corresponding to $\mathbf{w}_1$ and $\mathbf{v}_2$.  By construction, all syllables of $c$ and $c'$ commute. Hence there exist $r\geq 0$ and $\Gamma_0\in \mathrm{Cliq}(\Gamma,r)$ such that $c = c_1\ldots c_r$ and $c' = c_1'\ldots c_r'$ are reduced operators of the same type $s'_1\ldots s'_r\in \mathcal{W}_{\Gamma_0}$.
However, note that the product $cc'$ need not be a reduced operator but
 a linear combination of reduced operators after omitting zero terms: 
 \begin{equation}\label{equation; ab-decomposition}
 \begin{aligned}
     cc' = & (c_{1}c_{1}')^\circ\ldots (c_{r}c_{r}')^\circ + \sum_{j=1}^{r-1}\sum_{1\le i_1<\ldots<i_j\le r}\varphi_{s_{i_1}'}(c_{{i_1}}c_{{i_1}}')\ldots\varphi_{s_{i_j}'}(c_{{i_j}}c_{{i_j}}')\\
			& \times (c_{1}c_{1}')^\circ\ldots(c_{{i_1}}c_{{i_1}}')^{\circ\wedge}\ldots(c_{{i_j}}c_{{i_j}}')^{\circ\wedge}\ldots(c_{r}c_{r}')^\circ \\
			& + \varphi_{s_{1}'}(c_{{1}}c_{{1}}')\ldots\varphi_{s_{r}'}(c_{{r}}c_{{r}}')\cdot\mathbf{1},
 \end{aligned}
 \end{equation}
where $(c_{{i_j}}c_{{i_j}}')^{\circ\wedge}$ indicates that these terms are omitted.
If \((c_ic_i')^\circ\ne0\) for every \(i=1,\ldots,r\), then 
the reduced operator in the sum above which has the largest length is 
$(c_{1}c_{1}')^\circ\ldots (c_{r}c_{r}')^\circ$, and it is straightforward to see that 
$$\ell\big((c_{1}c_{1}')^\circ\ldots (c_{r}c_{r}')^\circ\big) = r\quad \text{and}\quad \ell\big(a'(c_{1}c_{1}')^\circ\ldots (c_{r}c_{r}')^\circ b'\big) = k+n-r.$$

We now consider the general case. If cancellations occur, write $a = dd'$ and $b = d{''}f$ with $d'd{''} = \mathbf{1}$. Since $\mathcal{M}_{\Gamma}$ is finite, we have $d{''} = d'^{-1}$. Using braid relations, we may assume that $\ell(d')$ is maximal among all such factorizations for which no cancellation occurs between $d$ and $f$. Then $a{''}=d'$. Applying the previous argument to $d$ and $f$, we obtain factorizations $d = a'c$ and $f = c'b'$ such that $\ell\big(a'(c_{1}c_{1}')^\circ\ldots (c_{r}c_{r}')^\circ b'\big) = k+n - 2\ell(a{''})-r$ and $c$, $c'$ satisfy (ii).  The result follows by setting $q = r+ 2\ell(a{''})$.
\end{proof}

\begin{rem}
In the special case where $\mathcal{M}_{\Gamma}$ is a graph product of group von Neumann algebras and $a,b$ are group unitaries, one has $cc'=(c_{1}c_{1}')^\circ\ldots (c_{r}c_{r}')^\circ$ which corresponds to the first term in \eqref{equation; ab-decomposition}, since $\varphi_s(\lambda_h) = 0$ for any $e\ne h\in G_s$ ($s\in V\Gamma$). Therefore, in this case,  the above proposition simplifies to Lemma \ref{lem;reduced-decompsition-groups}.
\end{rem}

As indicated by the previous proposition, the product of two reduced operators admits a factorization containing a commutative component (the $cc'$ term). To control this component, we assume throughout the remainder of this section that the graph $\Gamma$ satisfies the condition \eqref{eq; Fi1}, namely, there exists an integer $N\geq 0$ such that $|\mathrm{Link}(s)|\le N$ for all $ s\in V\Gamma$. 


To show that the Hilbert transforms introduced in Definition~\ref{defn;Hilbert transform} satisfy \eqref{eq;Cotlar-Pd} for some $d$ depending on the graph, we need to impose a compatibility condition on the coefficients $\varepsilon$:
For any clique $\Gamma_0\in \mathrm{Cliq}(\Gamma, \ge 2)$ with $V\Gamma_0 = \{s_1,\ldots, s_r\}$, we will  write $\varepsilon_{\Gamma_0}$ (resp. $\mathcal{L}_{\Gamma_0}$) as $\varepsilon_{s_1\ldots s_r}$ (resp. $\mathcal{L}_{s_1\ldots s_r}$). 
 For every integer $2\leq m\leq N$ and every connected non-clique induced subgraph $\Lambda\subseteq\Gamma$ with 
$ V\Lambda=\{s_1,\ldots,s_m,s_{m+1}\}$,
such that $\{s_1,\ldots,s_m\}$ generates a clique in $\Lambda$ and 
\[ \mathrm{Link}_{\Lambda}(s_m) = \{s_1,\ldots,s_{m-1},s_{m+1}\}, \] 
we require that 
\begin{equation}\label{eq: compatibility}
\varepsilon_{s_1\cdots s_m} = \varepsilon_{s_1\cdots s_{m-1}}.  
\end{equation}
We illustrate the conditions above with the following two examples of subgraphs that satisfy our compatibility conditions for $m=2$ and $m=3$.

\begin{figure}[!ht] 
\vspace{-5pt}
\centering 
\begin{minipage}[t]{0.49\textwidth} 
\centering 
\begin{tikzpicture}[ vertex/.style={circle, fill, inner sep=1pt} ] 
\path[use as bounding box] (-1.2,-0.9) rectangle (2.2,0);

\node[vertex, label=below:$s_1$] (s1) at (-1,0) {}; 
\node[vertex, label=below:$s_2$] (s2) at (0,0) {}; 
\node[vertex, label=below:$s_3$] (s3) at (1,0) {};

\draw (s1) -- (s2) -- (s3); 
\end{tikzpicture} 
\captionof{figure}{Subgraph $\Lambda_1$.} 
\label{fig:figure2} 
\end{minipage} 
\hfill 
\begin{minipage}[t]{0.49\textwidth} 
\centering 
\begin{tikzpicture}[ vertex/.style={circle, fill=black, inner sep=1pt} ] 

\node[vertex, label=left:$s_1$] (s1) at (0,0) {}; 
\node[vertex, label=below:$s_2$] (s2) at ({1/2},{-sqrt(3)/2}) {}; 
\node[vertex, label=below right:$s_3$] (s3) at (1,0) {}; 
\node[vertex, label=right:$s_4$] (s4) at (2,0) {}; 
\draw (s1) -- (s2) -- (s3) -- (s1); 
\draw (s3) -- (s4); 
\end{tikzpicture} 
\captionof{figure}{Subgraph $\Lambda_2$.} 
\label{fig:figure3} 
\end{minipage} 
\end{figure}
Applying the compatibility conditions to  $\Lambda_1$ yields  
$\varepsilon_{s_1s_2}=\varepsilon_{s_1} $ and $ \varepsilon_{s_2s_3}=\varepsilon_{s_3}$. Applying them to $\Lambda_2$ and its connected non-clique subgraphs gives $\varepsilon_{s_1s_2s_3} = \varepsilon_{s_1s_2}$, $\varepsilon_{s_1s_3} = \varepsilon_{s_1}$, $\varepsilon_{s_3s_4} = \varepsilon_{s_4}$ and $\varepsilon_{s_2s_3} = \varepsilon_{s_2}$.

For a graph satisfying \eqref{eq; Fi1}, we from now on restrict our attention to Hilbert transforms $H_\varepsilon$ satisfying the above compatibility conditions for every $2\leq m\leq N$. These compatibility conditions are necessary for the proof of the Cotlar identity, as illustrated by the following example.
 Consider the graph in Figure \ref{graph;1} and let $g = (g_{s_2}g_{s_1})^{M}g_{t_1}$ and $h = g_{t_1}$. Then 
 $\ell(gh^{-1}) = 2M$. 
 Without the compatibility conditions above, for any $d$ depending on $N$, we can choose $M>d$, then 
\eqref{eq;Cotlar-Pd} fails to hold, since
$$\varepsilon_{t_1s_2}\varepsilon_{t_1}\ne \varepsilon_{t_1s_2}\varepsilon_{s_1} + \varepsilon_{t_1}\varepsilon_{s_2} - \varepsilon_{s_2}\varepsilon_{s_1}.$$
However, imposing the compatibility conditions yields
  $$H_{\varepsilon} = \varepsilon_0\varphi_{\Gamma}+ \varepsilon_{s_1}L_{s_1} + \varepsilon_{t_1}(\mathcal{L}_{t_1} + \mathcal{L}_{s_2t_1} + \mathcal{L}_{s_2})+ \varepsilon_{t_2}L_{t_2}.$$
Thus, for a path graph such as the one in Figure~\ref{graph;1}, the resulting class of Hilbert transforms distinguishes only three components. For graphs with more branches or with cliques containing more than three vertices, however, these Hilbert transforms may distinguish a larger number of components.

Now we are ready to prove \ref{Thm A} for this class of Hilbert transforms, which says that they satisfy 
\eqref{eq;Cotlar-Pd} for $d=6N$.  

For convenience, given a reduced operator $a$ (resp. $g\in G_\Gamma$), we write  $\varepsilon_{\mathrm{MCP}(a)}$ (resp. $\varepsilon_{\mathrm{MCP}(g)}$)  in place of $\varepsilon_{\Gamma_0}$ when $\Gamma_0$ is the clique associated with $\mathrm{MCP}(a)$ (resp. $\mathrm{MCP}(g)$).

\begin{proof}[Proof of \ref{Thm A}]
	By linearity, it suffices to assume that $x$ and $y$ are reduced operators. 
    Suppose that $\ell(x)=k$ and $\ell(y)=n$. By Proposition \ref{prop; ab-decomposition}, $x$ and $y^*$ admit the following factorizations:
	$$x = a'ca''\quad \text{and}\quad y^* = a''^{-1}c'b',$$
	where $a'', c$ and $ c'$ are as in Proposition \ref{prop; ab-decomposition}. It follows from (\ref{equation; ab-decomposition}) that 
    \begin{equation*}
		\begin{aligned}
			xy^* =\, & (a'c)\cdot (c'b')\\
			 =\, &  a'(c_{1}c_{1}')^\circ\ldots (c_{r}c_{r}')^\circ b' + \sum_{j=1}^{r-1}\sum_{1\le i_1<\ldots<i_j\le r}\varphi_{s_{i_1}'}(c_{{i_1}}c_{{i_1}}')\ldots\varphi_{s_{i_j}'}(c_{{i_j}}c_{{i_j}}')\\
			& \times a'(c_{1}c_{1}')^{\circ}\ldots(c_{{i_1}}c_{{i_1}}')^{\circ\wedge}\ldots(c_{{i_j}}c_{{i_j}}')^{\circ\wedge}\ldots(c_{r}c_{r}')^\circ b'\\
			& + \varphi_{s_{1}'}(c_{{1}}c_{{1}}')\ldots\varphi_{s_{r}'}(c_{{r}}c_{{r}}')a'\cdot b':= \mathrm{I} + \mathrm{II}+ \mathrm{III}.
		\end{aligned}
	\end{equation*}
       Without loss of generality, we may assume that \((c_ic_i')^\circ\ne0\) for every \(i=1,\ldots,r\). 
    Note that $\mathrm{I}$ is a reduced operator, whereas the terms in $\mathrm{II}$ and $\mathrm{III}$ need not be reduced.   
    Write II further as a linear combination of reduced operators. 
    Note that if some terms in II satisfy 
    $$ a'(c_{1}c_{1}')^{\circ}\ldots(c_{{i_1}}c_{{i_1}}')^{\circ\wedge}\ldots(c_{{i_j}}c_{{i_j}}')^{\circ\wedge}\ldots(c_{r}c_{r}')^\circ b' = (c_{1}c_{1}')^{\circ}\ldots(c_{{i_1}}c_{{i_1}}')^{\circ\wedge}\ldots(c_{{i_j}}c_{{i_j}}')^{\circ\wedge}\ldots(c_{r}c_{r}')^\circ a'\cdot b',$$
    we move them to III. Therefore, we assume that II does not contain these terms.
     Write  the new III also as a linear combination of reduced operators.
     Observe that the maximal length of terms in this expansion is bounded by 
$\ell(\mathrm{I})=\ell(a'(c_{1}c_{1}')^\circ\ldots (c_{r}c_{r}')^\circ b')$. If $\ell(\mathrm{I}) \leq 6N$,  the identity \eqref{eq;Cotlar-ab} holds trivially. 
We may therefore assume that $\ell(\mathrm{I})>6N$.

It is enough to consider  the case $N\ge1$, since when $N = 0$, the graph $\Gamma$ has no edges, \ref{Thm A} reduces to \cite[Proposition 3.2(vi)]{mei2017free} for free products of finite von Neumann algebras.
Since $2N\ge N+1$, the inequality $\ell(\mathrm{I})>6N$ gives $\ell(\mathrm{I})>3N+3$.
By the assumption \eqref{eq; Fi1} on the graph $\Gamma$, we have 
$$\ell\big((c_{1}c_{1}')^\circ\ldots (c_{r}c_{r}')^\circ\big)\le N+1.$$
Therefore, 
 either $\ell(a')> N+1$ or $\ell(b')>N+1$.
    \medskip
    
    Assume first that $\ell(a')>N+1$. 
    We claim that $\varepsilon_{\mathrm{MCP}(x)} = \varepsilon_{\mathrm{MCP}(a')}$. Consequently, we have $H_{\varepsilon}(x) = H_{\varepsilon}(a')ca{''}$.
    If no syllable of $ca''$ belongs to $\mathrm{Init}(x)$, then $\mathrm{MCP}(x) = \mathrm{MCP}(a')$, and the claim follows. Otherwise, let $f\in \mathcal{M}^\circ_{s}$ be a syllable of $ca''$ belonging to $\mathrm{Init}(x)$. Then $f$ commutes with every syllable of $a'$. 
    Let $\Gamma_1$ be the clique associated with $\mathrm{MCP}(a')$ and $ S_{a'}:=\{s_{1},\ldots,s_{l}\}$ be the set of vertices of $a'$.
    Then 
    $$V\Gamma_1\subseteq S_{a'} \subseteq \mathrm{Link}(s),$$
    which implies that $\{s\}\cup V\Gamma_1$ generates a clique, and the subgraph generated by $\{s\}\cup S_{a'}$ is connected. Since $\ell(a')>N+1$, assumption \eqref{eq; Fi1} ensures that $\{s\}\cup S_{a'}$ does not generate a clique. Therefore,  the compatibility conditions \eqref{eq: compatibility} on $H_\varepsilon$ give 
    $$\varepsilon_{\mathrm{MCP}(a')} = \varepsilon_{\mathrm{MCP}(a'f)}.$$
    Applying this argument successively to all syllables of $ca''$ that belong to $\mathrm{Init}(x)$, we obtain 
    $$\varepsilon_{\mathrm{MCP}(a')} = \varepsilon_{\mathrm{MCP}(x)}.$$
    Therefore, $H_{\varepsilon}(x) = H_{\varepsilon}(a')ca{''}$. Moreover, applying a similar argument to $xy^*$ yields
     \begin{equation}\label{eq;expansion 1}
     \begin{aligned}
         H_{\varepsilon}(xy^*) &= H_{\varepsilon}(a')(c_{1}c_{1}')^\circ\ldots (c_{r}c_{r}')^\circ b' + \sum_{j=1}^{r-1}\sum_{1\le i_1<\ldots<i_j\le r}\varphi_{s_{i_1}'}(c_{{i_1}}c_{{i_1}}')\ldots\varphi_{s_{i_j}'}(c_{{i_j}}c_{{i_j}}')\\
			& \times H_{\varepsilon}(a')(c_{1}c_{1}')^\circ\ldots(c_{{i_1}}c_{{i_1}}')^{\circ\wedge}\ldots(c_{{i_j}}c_{{i_j}}')^{\circ\wedge}\ldots(c_{r}c_{r}')^\circ b'+H_{\varepsilon}(\mathrm{III}).
            \end{aligned}
     \end{equation}
     Similarly, if $\ell(b')>N+1$, then $H_{\varepsilon'}^{\mathrm{op}}(y^*) = a{''^{-1}}c'H_{\varepsilon'}^{\mathrm{op}}(b')$ and 
     \begin{equation}\label{eq;expansion 2}
         \begin{aligned}
         H_{\varepsilon'}^{\mathrm{op}}(xy^*) &= a'(c_{1}c_{1}')^\circ\ldots (c_{r}c_{r}')^\circ H_{\varepsilon'}^{\mathrm{op}}(b') + \sum_{j=1}^{r-1}\sum_{1\le i_1<\ldots<i_j\le r}\varphi_{s_{i_1}'}(c_{{i_1}}c_{{i_1}}')\ldots\varphi_{s_{i_j}'}(c_{{i_j}}c_{{i_j}}')\\
			& \times a'(c_{1}c_{1}')^\circ\ldots(c_{{i_1}}c_{{i_1}}')^{\circ\wedge}\ldots(c_{{i_j}}c_{{i_j}}')^{\circ\wedge}\ldots(c_{r}c_{r}')^\circ H_{\varepsilon'}^{\mathrm{op}}(b')+H_{\varepsilon'}^{\mathrm{op}}(\mathrm{III}).
            \end{aligned}
     \end{equation}
     
     We now treat $\mathrm{III}$. If all terms in the expansion of $a'\cdot b'$ have length at most $6N$, then using the facts that $H_{\varepsilon}P_{>6N} = P_{>6N}H_{\varepsilon}$ and $H_{\varepsilon'}^{\mathrm{op}}P_{>6N} = P_{>6N}H_{\varepsilon'}^{\mathrm{op}}$, equation \eqref{eq;expansion 1} implies 
     \begin{equation}\label{eq;8}
      P_{>6N} \big(H_{\varepsilon}(x)y^*\big)=P_{>6N}H_{\varepsilon}(xy^*)   
     \end{equation}
     and 
     $$P_{>6N}H_{\varepsilon'}^{\mathrm{op}}\big(H_{\varepsilon}(x)y^*\big) = P_{>6N}H_{\varepsilon'}^{\mathrm{op}}H_{\varepsilon}(xy^*).$$
    Replacing $y^*$ in \eqref{eq;8} by $H_{\varepsilon'}^{\mathrm{op}}(y^*)$, we obtain
     $$P_{>6N}\big(H_{\varepsilon}(x)H_{\varepsilon'}^{\mathrm{op}}(y^*)\big) = P_{>6N}H_{\varepsilon}\big(xH_{\varepsilon'}^{\mathrm{op}}(y^*)\big).$$
 Therefore, \eqref{eq;Cotlar-ab} holds. Similarly, by the facts that $H_{\varepsilon'}^{\mathrm{op}}P_{>6N} = P_{>6N}H_{\varepsilon'}^{\mathrm{op}}$ and $H_{\varepsilon}H_{\varepsilon'}^{\mathrm{op}}=H_{\varepsilon'}^{\mathrm{op}}H_{\varepsilon}$, equation \eqref{eq;expansion 2} implies 
 \begin{equation}\label{eq;9}
      P_{>6N} \big(xH_{\varepsilon'}^{\mathrm{op}}(y^*)\big)=P_{>6N}H_{\varepsilon'}^{\mathrm{op}}(xy^*)   
     \end{equation}
 and 
 \begin{equation*}
      P_{>6N}H_{\varepsilon} \big(xH_{\varepsilon'}^{\mathrm{op}}(y^*)\big)=P_{>6N}H_{\varepsilon'}^{\mathrm{op}}H_{\varepsilon}(xy^*).   
     \end{equation*}
Replacing $x$ in \eqref{eq;9} by $H_{\varepsilon}(x)$, we obtain
$$
P_{>6N} \big(H_{\varepsilon}(x)H_{\varepsilon'}^{\mathrm{op}}(y^*)\big)=P_{>6N}H_{\varepsilon'}^{\mathrm{op}}\big(H_{\varepsilon}(x)y^*\big).
$$
Hence, \eqref{eq;Cotlar-ab} follows. 
 
 If instead the expansion of $a'\cdot b'$ contains terms of length  greater than $6N$, we apply Proposition \ref{prop; ab-decomposition} to $a'$ and $b'$ to obtain the following factorizations 
     $$a' = a'_2da_2''\quad \text{and}\quad b' = a_2''^{-1}d'b'_2.$$  
Using an argument similar to that for $x$ and $y$, it is straightforward to see that (\ref{eq;Cotlar-ab}) holds if  all reduced operators in the linear span of $a'_2\cdot b'_2$ have length at most $6N$. 
     If not, we continue decomposing $a'_2$ and $b'_2$  using Proposition \ref{prop; ab-decomposition} and 
     repeat the above argument.  After finitely many steps, we obtain some $m\geq 2$ such that all reduced operators in the linear span of $a'_m\cdot b'_m$ have length at most $6N$. The conclusion then follows. 
\end{proof}

\medskip

In the special case where $\mathcal{M}_{\Gamma}$ is the von Neumann algebra associated with a graph product of groups $G_{\Gamma}$, with $\Gamma$ a simplicial graph satisfying~\eqref{eq; Fi1} and each $G_s$ ($s \in V\Gamma$) a discrete group, we adopt a more direct approach to establish the Cotlar identity. 
We exploit Lemma~\ref{lem;reduced-decompsition-groups} in place of Proposition~\ref{prop; ab-decomposition}, which leads to a simpler proof of the generalized Cotlar identity.

\begin{lem}\label{lem; Cotlar-gp}
	For any  $g, h\in G_{\Gamma}$ with $\ell(gh^{-1})>6N$, we have either
	\begin{equation*}
		H_{\varepsilon}(\lambda_g\lambda_h^*) = H_{\varepsilon}(\lambda_g)\lambda_h^*\quad \text{or}\quad H_{\varepsilon}^{\mathrm{op}}(\lambda_g\lambda_h^*) = \lambda_gH_{\varepsilon}^{\mathrm{op}}(\lambda_h^*).
	\end{equation*}
\end{lem}
\begin{proof}
	By Lemma \ref{lem;reduced-decompsition-groups}, we can write 
$$g = g_1g_2'g_3\quad \text{and}\quad h^{-1} = g_3^{-1}g_2''h_1,$$
	with $g_2', g_2''\in G_{\Gamma_0}$ for some $\Gamma_0\in \mathrm{Cliq}(\Gamma,r)$, and $\ell(g_2') = \ell(g_2'') = \ell(g_2'g_2'') = r\le N+1$. It follows that $$\lambda_g\lambda_h^* = \lambda_{g_1}\lambda_{g_2'g_2''}\lambda_{h_1}.$$
    We first consider the case $N = 0$. Since $$0 = 6N<\ell(gh^{-1})\le\ell(g_2')+\ell(gh^{-1}) = \ell (g_1g_2')+ \ell(g_2''h_1),$$ 
we have either $\ell(g_1g_2')>0$ or $\ell(g_2''h_1)>0$.
	Suppose first that $\ell(g_1g_2')>0$.  The bound $|\mathrm{Link}(s)|\le 0$ for all $s\in V\Gamma$ 
    ensures that no syllable in the reduced expression of 
$g_3$ (resp. $h_1$) can be moved to the leftmost position of
$g$ (resp. $gh^{-1}$) via braid relations in the underlying right-angled Coxeter group. Consequently, 
\begin{equation*}
    \mathrm{MCP}(g) = \mathrm{MCP}(g_1g_2')\quad \text{and}\quad \mathrm{MCP}(gh^{-1})= \mathrm{MCP}(g_1g_2'g_2'')
\end{equation*}
and then
\begin{equation}\label{eq;proof-gp-case1}
    H_{\varepsilon}(\lambda_g) = H_{\varepsilon}(\lambda_{g_1}\lambda_{g_2'})\lambda_{g_3} \quad \text{ and } \quad H_{\varepsilon}(\lambda_g\lambda_h^*) = H_{\varepsilon}(\lambda_{g_1}\lambda_{g_2'g_2''})\lambda_{h_1}.
\end{equation}
    Since $g_2',g_2''\in G_{\Gamma_0}$ and $\ell(g_2') = \ell(g_2'') = \ell(g_2'g_2'') = r$, we have $H_{\varepsilon}(\lambda_{g_1}\lambda_{g_2'g_2''}) = H_{\varepsilon}(\lambda_{g_1}\lambda_{g_2'})\lambda_{g_2''}$.
Combining this identity with \eqref{eq;proof-gp-case1} yields
$$H_{\varepsilon}(\lambda_g\lambda_h^*) = H_{\varepsilon}(\lambda_g)\lambda_h^*.$$

	Suppose now that $\ell(g_2''h_1)>0$. By the same argument, no syllable in the reduced expression of  $g_3^{-1}$ can be moved to the rightmost position of $h^{-1}$. Therefore, 
	$$H_{\varepsilon}^{\mathrm{op}}(\lambda_h^*) = \lambda_{g_3^{-1}} H_{\varepsilon}^{\mathrm{op}}(\lambda_{g_2''}\lambda_{h_1})\quad \text{ and } \quad H_{\varepsilon}^{\mathrm{op}}(\lambda_g\lambda_h^*) = \lambda_gH_{\varepsilon}^{\mathrm{op}}(\lambda_h^*).$$	

\medskip

    Next, we consider the case $N\ge 1$.
Since $\ell(gh^{-1}) =  \ell(g_1g_2'g_2''h_1) >6N$ and $2N\ge N+1$, we have 
$$\ell(g_1g_2'g_2''h_1)>3N+3.$$
Moreover, by \eqref{eq; Fi1}, $\ell(g_2') = \ell(g_2'') = \ell(g_2'g_2'')\le N+1$. Then either 
$$\ell(g_1)>N+1\quad \text{or}\quad \ell(h_1)>N+1.$$
There are four cases, according to whether syllables of \(g_3\) occur in \(\mathrm{Init}(g)\) and in \(\mathrm{Init}(h)\). The arguments for the four cases are analogous; we therefore consider only the case in which no syllable of \(g_3\) belongs to \(\operatorname{Init}(g)\), whereas at least one syllable of \(g_3\) belongs to \(\operatorname{Init}(h)\).   Assume that $g_2''h_1$ has type $s_{21}\ldots s_{2r}s_{11}\ldots s_{1j}$.

    Since no syllable of $g_3$ belongs to $\mathrm{Init}(g)$, we have
    $\mathrm{MCP}(g) = \mathrm{MCP}(g_1g_2')$ and then 
    $$H_{\varepsilon}(\lambda_g)\lambda_{g_3^{-1}} = H_{\varepsilon}(\lambda_{g_1}\lambda_{g_2'}).$$
Again since $g_2',g_2''\in G_{\Gamma_0}$ and $\ell(g_2') = \ell(g_2'') = \ell(g_2'g_2'')$, we further obtain 
\begin{equation}\label{eq;proof-gp-case11}
H_{\varepsilon}(\lambda_{g_1}\lambda_{g_2'g_2''}) = H_{\varepsilon}(\lambda_{g_1}\lambda_{g_2'})\lambda_{g_2''} = H_{\varepsilon}(\lambda_g)\lambda_{g_3^{-1}}\lambda_{g_2''}.  
\end{equation}
Let $f_k\in G_s$ be a syllable of $g_3$ belonging to $\mathrm{Init}(h)$, and let $S_{h_1^{-1}g_2''^{-1}} := \{s_{11},\ldots,s_{1j},s_{21},\ldots,s_{2r}\}$, where repeated vertices are omitted. Then $f_k$ commutes with every syllable of $h_1^{-1}g_2''^{-1}$, which implies that $$S_{h_1^{-1}g_2''^{-1}} \subseteq\mathrm{Link}(s)$$ 
and then the subgraph generated by $\{s\}\cup S_{h_1^{-1}g_2''^{-1}}$ is connected.

Suppose first that $\ell(h_1)>N+1$. By \eqref{eq; Fi1}, the set $S_{h_1^{-1}g_2''^{-1}}$ does not generate a clique. Therefore,  the defining conditions on $H_\varepsilon$ give $\varepsilon_{\mathrm{MCP}(h_1^{-1}g_2''^{-1}f_k)} = \varepsilon_{\mathrm{MCP}(h_1^{-1}g_2''^{-1})}$. 
Applying the same argument successively to all syllables of $g_3$ that belong to $\mathrm{Init}(h)$, we have $\varepsilon_{\mathrm{MCP}(h)} = \varepsilon_{\mathrm{MCP}(h_1^{-1}g_2''^{-1})}$ and then 
$$H_{\varepsilon}^{\mathrm{op}}(\lambda_{h}^*) = \lambda_{g_3^{-1}}H_{\varepsilon}^{\mathrm{op}}(\lambda_{g_2''}\lambda_{h_1}).$$
Using the fact that $g_2',g_2''\in G_{\Gamma_0}$ and $\ell(g_2') = \ell(g_2'') = \ell(g_2'g_2'')$, we have
\begin{equation}\label{eq;proof-gp-Cotlar-eq2}
    H_{\varepsilon}^{\mathrm{op}}(\lambda_{g_2'}\lambda_{g_2''}\lambda_{h_1}) = \lambda_{g_2'}H_{\varepsilon}^{\mathrm{op}}(\lambda_{g_2''}\lambda_{h_1})=  \lambda_{g_2'}\lambda_{g_3}H_{\varepsilon}^{\mathrm{op}}(\lambda_{h}^*).
\end{equation}
It remains to examine the effect of the syllables of $g_1$ on $g_2'g_2''h_1$. If no syllable of $g_1$ belongs to $\mathrm{Init}(g_2'g_2''h_1)$, then \eqref{eq;proof-gp-Cotlar-eq2} implies that
$$H_{\varepsilon}^{\mathrm{op}}(\lambda_g\lambda_{h}^*) =  \lambda_{g_1}H_{\varepsilon}^{\mathrm{op}}(\lambda_{g_2'}\lambda_{g_2''}\lambda_{h_1}) = \lambda_{g}H_{\varepsilon}^{\mathrm{op}}(\lambda_{h}^*).$$
Suppose instead that a syllable $f_m\in G_t$ of $g_1$ belongs to $\mathrm{Init}(g_2'g_2''h_1)$. Then $S_{g_2'g_2''h_1}\subseteq \mathrm{Link}(t)$ and the subgraph generated by $\{t\}\cup S_{g_2'g_2''h_1}$ is connected. Since $\ell(h_1)>N+1$, by the similar argument, $\varepsilon_{\mathrm{MCP}(gh^{-1})} = \varepsilon_{\mathrm{MCP}(g_2'g_2''h_1)}$. Combining this identity with \eqref{eq;proof-gp-Cotlar-eq2}, we obtain
$$H_{\varepsilon}^{\mathrm{op}}(\lambda_g\lambda_{h}^*) =  \lambda_{g_1}H_{\varepsilon}^{\mathrm{op}}(\lambda_{g_2'}\lambda_{g_2''}\lambda_{h_1}) = \lambda_{g}H_{\varepsilon}^{\mathrm{op}}(\lambda_{h}^*).$$

We now treat the case $\ell(g_1)>N+1$. In this case, it suffices to prove that $H_{\varepsilon}(\lambda_g\lambda_{h}^*) = H_{\varepsilon}(\lambda_g)\lambda_h^*$. By \eqref{eq;proof-gp-case11}, it remains only to examine the effect of the syllables of $h_1$ on $g_1g_2'g_2''$. By an argument similar to the one above, we have
$$H_{\varepsilon}(\lambda_g\lambda_{h}^*) = H_{\varepsilon}(\lambda_{g_1}\lambda_{g_2'g_2''})\lambda_{h_1} = H_{\varepsilon}(\lambda_g)\lambda_h^*.$$
This completes the proof.
\end{proof}


\begin{cor}\label{cor; Cotlar-identity-group-vNA}
	Let $g,h\in G_{\Gamma}$  with $\ell(gh^{-1})>6N$ and $\varepsilon, \varepsilon'$ be two sequences in $\{\pm1\}^{\#\mathrm{Cliq}(\Gamma)}$ satisfying the compatibility conditions. Then we have the following Cotlar identity
	\begin{equation*}
		H_{\varepsilon}(\lambda_g)H_{\varepsilon'}^{\mathrm{op}}(\lambda_h^*) = 	H_{\varepsilon}\big(\lambda_g	H_{\varepsilon'}^{\mathrm{op}}(\lambda_h^*)\big) + 	H_{\varepsilon'}^{\mathrm{op}}\big(	H_{\varepsilon}(\lambda_g)\lambda_h^*\big)- 	H_{\varepsilon'}^{\mathrm{op}}	H_{\varepsilon}(\lambda_g\lambda_h^*).
	\end{equation*}
\end{cor}





\section{Haagerup-type inequalities on graph products of finite von Neumann algebras}\label{sec; Haagerup-type-inequality}

In this section, we provide equivalent characterizations of the hypothesis on  $P_{\le 6N}$ in Theorem \ref{thm;Hilbert-trans-finite-vnA}.
We then apply Theorem~\ref{thm;Hilbert-trans-finite-vnA} to graph products of finite-dimensional von Neumann algebras to obtain \ref{Thm B}. These provide a rich class of examples on which we can study the boundedness of Hilbert transforms, including  von Neumann algebras associated with graph products of finite groups, right-angled Hecke von Neumann algebras, and graph products of finite quantum groups.

Throughout this section, let $\Gamma$ be a simplicial graph satisfying~\eqref{eq; Fi1}, and assume further that $N \geq 1$ in~\eqref{eq; Fi1}. Note that when $N = 0$, the graph product reduces to a free product. In this case, the Cotlar identity, together with properties of conditional expectations, implies the boundedness of Hilbert transforms \cite{mei2017free}, and no assumption about the projection is needed.

\begin{prop}\label{thm;equivalence-Pd-Haagerup-type}
    Let  $\mathcal{M}_{\Gamma}$ be a graph product of finite von Neumann algebras associated with $\Gamma$.  
    Then the following statements are equivalent.
    \begin{enumerate}
        \item[(i)] The projection $P_{ \le 6N}$ is bounded from $L_2(\mathcal{M}_{\Gamma})$ to $\mathcal{M}_{\Gamma}$.
        
        \item[(ii)]  $\mathcal{M}_{\Gamma}$ is a graph product of finite-dimensional von Neumann algebras such that 
        \begin{equation*}
 \sup_{s\in V\Gamma}\|\mathrm{id}:L^2(\mathcal{M}_s,\varphi_s)\to \mathcal{M}_s\|<\infty.
        \end{equation*}

        \item[(iii)] $\mathcal{M}_{\Gamma}$ satisfies a Haagerup-type inequality, that is,   $P_d$  is bounded from  $L_2(\mathcal{M}_{\Gamma})$ to $\mathcal{M}_{\Gamma}$ for any $d\geq 1$: 
\begin{equation}\label{eq: RD for M}
\| P_d (x) \| \leq C_d \|x\| _{L_2(\mathcal{M}_{\Gamma})}.
\end{equation}
    \end{enumerate}
\end{prop}

Heuristically, the proposition above states that a ``local piecewise" Haagerup-type inequality implies the ``global" one when the block length on graph products is considered. This result differs from \cite[Theorem 4.1]{Laura2013} even in the discrete group case. 
Moreover, this generalizes the fact that right-angled Hecke von Neumann algebras associated with finite graphs satisfy the Haagerup-type inequality in \cite[Theorem 3.4]{caspers2021graph}, and provides a new proof that avoids the use of Khintchine-type inequalities.

 \medskip

The main difficulty in the above proposition is the implication (ii)$\implies$(iii). We begin with a few observations.
Note that for any $s \in V\Gamma$ and any $\Gamma_0 \in \mathrm{Cliq}(\Gamma, \ge 2)$, the algebras $\mathcal{M}_s$ and $\mathcal{M}_{\Gamma_0}$ satisfy the Haagerup-type inequality in (iii) since each $\mathcal{M}_s$ is finite dimensional, 
and by \eqref{eq; Fi1} each $\mathcal{M}_{\Gamma_0}$ is finite dimensional. It follows that for any $\Gamma' \in \mathrm{Cliq}(\Gamma, \ge 1)$ there exists a polynomial $\phi_{\Gamma'}$ such that for all $x \in \mathcal{M}_{\Gamma'}$ and all $d \ge 1$,
$\|P_d(x)\|\le \phi_{\Gamma'}(d)\|x\|_2$.


The following lemma plays a key role in establishing \eqref{eq: RD for M}. 
The extension  of $P_d$ as an orthogonal projection from $L_2(\mathcal{M}_\Gamma) = \mathcal{H}$ onto the closure of the linear span of $\mathcal{H}^{\circ}_{\mathbf{w}}$ with $\ell(\mathbf{w}) = d$ will  still be denoted by $P_d$. 

\begin{lem}\label{prop;Pm-Pl-finite-dim}
	Let $\Gamma$ be a simplicial graph and 
    $\mathcal{N}_{\Gamma}$ be the associated graph product of finite von Neumann algebras. Let $k,l,m\geq 1$ and $x\in P_k(\mathcal{N}_{\Gamma})$. If $P_mxP_l\ne 0$, then $|m-k|\le l\le m+k$.
\end{lem}

\begin{proof}
    By linearity, we may assume that $x \in \mathcal{N}_k$ is a reduced operator.
 By \cite[Lemma 2.5]{caspers2021graph}, 
 $x$ can be written as a finite sum of terms, each of which is a product of creation, diagonal, and annihilation operators. More precisely, each such term consists of $i$ creation operators, $r$ diagonal operators, and $k - (i+r)$ annihilation operators, where $0 \leq i \leq k-r$ and $0 \leq r \leq k$.

Let $\xi \in \mathcal{H}$ be such that $P_m x P_l \xi \neq 0$. By linearity we may assume that $P_l \xi$ is an elementary tensor in some $\mathcal{H}_{\mathbf{w}}^{\circ}$ with $\ell(\mathbf{w}) = l$.
Then there exists at least one term in the above decomposition whose action on $P_l \xi$ is nonzero. For such a term, the annihilation operators remove $k - (i+r)$ tensor components, the diagonal operators preserve the length, and the creation operators add $i$ tensor components. Hence the resulting vector has length
$$
m=l-(k-(i+r))+i=l-k+2i+r.$$
Moreover, since this term does not vanish, after the annihilation step there must remain at least $r$ tensor components for the diagonal operators to act, which yields
$$ l-(k - (r+i))\geq r,$$
that is, $l\geq k-i$. Combining these observations, we obtain
 $l=m+k-2i-r$ and $l\geq k-i$. Since $0 \leq i \leq k-r$ and $0 \leq r \leq k$, it follows that $|m-k|\le l\le m+k$.
\end{proof}

The next lemma provides an equivalent characterization of the inequality~\eqref{eq: RD for M} on $\mathcal{M}_{\Gamma}$. Its proof follows the same strategy as that of \cite[Proposition 5.7]{caspers2017graph}, with minor  modifications. Therefore we omit the proof. However, we highlight several points relevant to our setting in the following.

First, under condition~\eqref{eq; Fi1}, the number of terms appearing in the decomposition of a reduced operator of length $k$ into a finite sum of products of creation, diagonal, and annihilation operators, as in \cite[Proposition 2.5]{caspers2021graph}, is bounded by a polynomial in $k$. Indeed, it follows from \eqref{eq; Fi1} that the maximal length of the diagonal part appearing in each term is  $k' = \min\{k,N+1\}$. Moreover, for each $i= 0,\ldots, k'$, the number of terms containing diagonal operators of length $i$ is also bounded polynomially in $k$. For instance, consider the terms without diagonal operators and containing exactly one creation operator. In this case, there are at most
$\binom{k'}{1}$ possible ways to move the creation operator to the first position, as in \cite[Proposition 2.5]{caspers2021graph}. The remaining cases are treated similarly and likewise yield polynomial bounds in $k$. Following the proof of \cite[Proposition 5.7]{caspers2017graph}, and replacing the constant $M$ appearing in the last line of their proof by a polynomial in $k$, up to a constant depending only on $N$, yields the desired result.

Furthermore, although \cite[Proposition 5.7]{caspers2017graph} is formulated in the framework of compact quantum groups, 
the argument carries over to our setting. 
More precisely, in \cite[Proposition 5.7]{caspers2017graph}, the assumption that each discrete quantum group $\widehat{\mathbb{G}}_{\Gamma_0}$  associated with a clique $\Gamma_0\in \mathrm{Cliq}(\Gamma)$ has property (RD) yields the estimate  $\|x'\xi'\|_2\le c_{\Gamma_0}\|x'\|_2\|\xi'\|_2$, where $x'$ is a diagonal operator in the compact quantum group associated with $\Gamma_0$ and $\xi'$ is a vector in $\mathcal{H}$. 
Since cliques in $\Gamma$ include singleton subgraphs, 
this implies that each vertex compact quantum group is finite. 
Moreover, while \cite[Proposition 5.7]{caspers2017graph} assumes that $\Gamma$ is finite, the proof remains valid under our assumptions that $\Gamma$ is a simplicial graph satisfying~\eqref{eq; Fi1} and that $\mathcal{M}_{\Gamma}$ satisfies~\eqref{eq;M-finite-dimentional-condition}. The latter assumption is necessary when $\Gamma$ is an infinite graph.

\begin{lem}\label{prop;Equi-char-finite-dim-RD-1}
Let $\Gamma$ be a simplicial graph satisfying (\ref{eq; Fi1}) and $\mathcal{M}_{\Gamma}$ be the graph product of finite-dimensional von Neumann algebras satisfying (\ref{eq;M-finite-dimentional-condition}). 
	Let $k,l,m\geq 1$. There exists a polynomial $\psi$ such that for every $x\in P_k(\mathcal{M}_{\Gamma})$, one has
	$$\|P_mxP_l\|\le \psi(k)\|x\|_2\quad \text{if } |k-l|\le m\le k+l,$$
    and $\|P_mxP_l\| =0 $ otherwise.
\end{lem}

We state an equivalent characterization of the above lemma below, which shows that $P_k$ is  bounded from  $L_2(\mathcal{M}_{\Gamma})$ to $\mathcal{M}_{\Gamma}$ for any $k\in\mathbb{N}$. In particular, $\mathcal{M}_\Gamma$ satisfies \eqref{eq: RD for M}.
\begin{lem}\label{prop;finite-dim-Pd-2-infty}
	 We keep the assumptions from the above lemma. The following statements are equivalent.
	\begin{enumerate}
		\item[(i)] There exists a polynomial $\psi$ such that for every $x\in P_k(\mathcal{M}_{\Gamma})$, one has
		$$\|P_mxP_l\|\le \psi(k)\|x\|_2\quad \text{if } |k-l|\le m\le k+l,$$
		and $\|P_mxP_l\| =0 $ otherwise. 
		
		\item[(ii)] 
        There exists a polynomial $\psi$ such that for any $x\in P_k(\mathcal{M}_{\Gamma})$, we have
        $$\|x\|\le \psi(k)\|x\|_2.$$
	\end{enumerate}
\end{lem}
\begin{proof}
    The implication (ii)$\implies$(i) is immediate. It remains to prove  (i)$\implies$(ii). 
    Let $x\in P_k(\mathcal{M}_{\Gamma})$ and $\xi\in \mathcal{H}$. By Lemma \ref{prop;Pm-Pl-finite-dim} and (i), 
\begin{align*}
 \|x\xi\|_2^2 \le \sum_m\Big(\sum_l\|P_mxP_l\xi\|_2\Big)^2 &=\sum_m\Big(\sum_{|m-k|\le l\le m+k}\|P_mxP_l\xi\|_2\Big)^2\\
   &\le \psi(k)^2\|x\|_2^2\sum_m\Big(\sum_{|m-k|\le l\le m+k}\|P_l\xi\|_2\Big)^2.  
\end{align*}  
   Applying the Cauchy-Schwarz inequality, we have 
    $$\sum_m\Big(\sum_{|m-k|\le l\le m+k}\|P_l\xi\|_2 \Big)^2\le (2k+1)\sum_m\sum_{|m-k|\le l\le m+k}\|P_l\xi\|_2^2\le (2k+1)^2\sum_l\|P_l\xi\|_2^2.$$
   Since $\sum_l\|P_l\xi\|_2^2 = \|\xi\|_2^2$, it follows that 
    $$\|x\xi\|_2 \le (2k+1)\psi(k)\|x\|_2\|\xi\|_2,\quad  \forall \xi\in \mathcal{H}.$$
     This proves (ii).
\end{proof}

In general, the projection $P_k$ is not completely bounded from  $L_2(\mathcal{M}_{\Gamma})$ to $\mathcal{M}_{\Gamma}$ if the graph is infinite. However, it is completely bounded when the graph is finite.

\begin{prop}\label{prop;cb-Pk}
    Let $\Gamma$ be a finite simplicial graph,
    and let $\mathcal M_\Gamma$ be the graph product of finite-dimensional von Neumann
    algebras. For every $k\ge 1$, the projection $P_k$ is completely bounded from $L_2(\mathcal{M}_{\Gamma})$ to $\mathcal{M}_{\Gamma}$. More precisely, there is $r_k>0$ such that
    $$\| P _k \|_{{\rm cb},2\rightarrow\infty}\leq r_k.$$
\end{prop}
\begin{proof}
    Let $E_k = P_k\big(L_2(\mathcal{M}_\Gamma)\big)$. Since each vertex von Neumann algebra is finite dimensional and the graph is finite, so is $E_k$. Let $b_1,\ldots,b_{m_k}\in\mathcal{M}_\Gamma$ be an $L_2$-orthonormal basis  of $E_k$. 
    For $1\le l\le m_k$, define $f_l:L_2(\mathcal M_\Gamma)\rightarrow\mathbb C$ by $f_l(x)=\varphi_{\Gamma}(b_l^*x)$.
    For every $x\in L_2(\mathcal{M}_\Gamma)$, we then have $P_k(x) = \sum_{l=1}^{m_k}f_l(x)b_l$. 
    By H\"older's inequality and the fact that $f_l(b_l) = 1$, we have $\|f_l\| = 1$ and then $\|f_l\|_{\mathrm{cb}} = 1$ for every $1\le l\le m_k$.
    Fix $n\ge1$ and consider $X=(x_{ij})_{n\times n}\in \mathbb M_n(L_2(\mathcal M_\Gamma))$. Set $F_l=\big(f_l(x_{ij})\big)_{i,j=1}^n\in \mathbb M_n$. It then follows that 
    $$ \|F_l\|_{\mathbb M_n} = \sup_{\substack{u,v\in\mathbb{C}^n\\ \|u\|_2,\|v\|_2\leq 1}}|u^*F_l v| = \sup_{\substack{u,v\in\mathbb{C}^n\\\|u\|_2,\|v\|_2\leq 1}}|f_l(u^*Xv)|
		\le\|X\|_{\mathbb M_n(L_2(\mathcal M_\Gamma))}.$$
        Consequently, we have
\begin{equation*}
    \begin{aligned}
        \|(\mathrm{id}_{\mathbb M_n}\otimes P_k)(X)\|_{\mathbb M_n(\mathcal{M}_\Gamma)} & = \left\| \sum_{l=1}^{m_k}\big(f_l(x_{ij})\big)\otimes b_l\right\|_{\mathbb M_n(\mathcal{M}_\Gamma)}\le \sum_{l=1}^{m_k}\left\|F_l\right\|_{\mathbb M_n}\|b_l\|_{\mathcal{M}_{\Gamma}}\\
        &\le \left(\sum_{j=1}^{m_k}\|b_j\|\right)\|X\|_{\mathbb M_n(L_2(\mathcal{M}_\Gamma))},
    \end{aligned}
\end{equation*}
It follows from Proposition \ref{thm;equivalence-Pd-Haagerup-type} (iii) that $\|b_j\|\le C_k$. Let $r_k := m_k C_k$.  We conclude that
	\[\|P_k\|_{\mathrm{cb},\,2\to\infty}
		\le r_k.\]
\end{proof}

Now we give the proof of Proposition \ref{thm;equivalence-Pd-Haagerup-type}.
 \begin{proof}[Proof of Proposition \ref{thm;equivalence-Pd-Haagerup-type}] 
For (i)$\implies$(ii), note that $P_1 = P_{\le 6N} \circ P_1$, and that $P_1$ is bounded on $L_2(\mathcal{M}_{\Gamma})$. Hence, if $\| P_{\leq 6N} (x) \| \leq C \|x\| _{L_2(\mathcal{M}_{\Gamma})}$,  then for any $s \in V\Gamma$ and any $y \in \mathcal{M}_s^\circ$, one has
 $$\|y\| = \|P_1 y\|\le C\|y\|_2.$$ 
 This, together with the decomposition $y = \varphi_s(y)\cdot \mathbf{1} + y^\circ$, implies that $\mathcal{M}_s$ is finite dimensional for all $s\in V\Gamma$, and the uniform upper bound for $\mathrm{id}: L_2(\mathcal{M}_{s})\to \mathcal{M}_{s}\, (s\in V\Gamma)$ is a constant depending on $C$.
 
 For (ii)$\implies$(iii), by Lemma \ref{prop;finite-dim-Pd-2-infty} (ii), there exists a polynomial $\psi$ such that
    $$\|P_d(x)\|\le \psi(d)\|x\|_2, \quad \text{for all } d\geq 1, x\in \mathcal{M}_{\Gamma}.$$
    By density, $P_{ d}$ extends to a bounded operator from $L_2(\mathcal{M}_\Gamma)$ to $\mathcal{M}_\Gamma$. 
    
 Finally, (iii) $\implies$(i) is immediate.
\end{proof}


By \ref{Thm A} and Proposition \ref{thm;equivalence-Pd-Haagerup-type}, graph products of finite-dimensional von Neumann algebras satisfying (\ref{eq;M-finite-dimentional-condition}) satisfy the assumptions in Theorem \ref{thm;Hilbert-trans-finite-vnA}. 

\begin{proof}[Proof of \ref{Thm B}]

For $2\le p<\infty$, the asserted $L_p$-boundedness is an immediate consequence of \ref{Thm A}, Theorem \ref{thm;Hilbert-trans-finite-vnA} and Proposition \ref{thm;equivalence-Pd-Haagerup-type}. The remaining range $1<p<2$ follows from duality.
\end{proof}

\begin{rem}\label{rem; no cb}
The bound for the Hilbert transforms in \ref{Thm B} depends only on $p$ and $N$. In particular,  it is independent of the size and clique structure of $\Gamma$ (compare Theorem \ref{thm;finite-graph-Lp-Hilbert-transform}).
The completely bounded version of \ref{Thm B} remains valid when $\Gamma$ is finite. This is an immediate consequence of Proposition \ref{prop;cb-Pk}, which states that for every $d\geq 1$, $P_{d}$ is completely bounded from $L_2(\mathcal{M}_{\Gamma})$ to $\mathcal{M}_{\Gamma}$. Then interpolation yields that $P_{d}$ is completely bounded from $L_2(\mathcal{M}_{\Gamma})$ to $L_q(\mathcal{M}_{\Gamma})$ for every $q>2$. Consequently, \ref{Thm A} implies that both $H_\varepsilon$ and $H_\varepsilon^{\mathrm{op}}$ are completely bounded.

\end{rem}


 Now we look at a direct consequence of \ref{Thm B}. Define an equivalence relation on $\mathrm{Cliq}(\Gamma, \ge 1)$ by declaring that two cliques $\Gamma_0,\Gamma_1\in \mathrm{Cliq}(\Gamma, \ge 1)$ are equivalent, i.e. $\Gamma_0 \sim \Gamma_1$, if $\varepsilon_{\Gamma_0} = \varepsilon_{\Gamma_1}$ under the compatibility conditions.
Let $J$ be the set of equivalence classes and $j\in J$.
Set $\mathcal{L}_{j} = \sum_{\Gamma'\sim j}\mathcal{L}_{\Gamma'}$  and $\mathcal{R}_j = \sum_{\Gamma'\sim j}\mathcal{R}_{\Gamma'}$. Let $\mathcal{L}_0 = \varphi_\Gamma(\cdot)\mathbf{1}$. Then the Hilbert transforms satisfying the compatibility conditions can be written as 
$$H_\varepsilon = \sum_{j\in J\cup \{0\}}\varepsilon_j \mathcal{L}_j\quad \text{and}\quad H_\varepsilon^{\mathrm{op}} = \sum_{j\in J\cup\{0\}}\varepsilon_j \mathcal{R}_j.$$
Applying \ref{Thm B} and the noncommutative Khintchine inequality in \cite[Th\'eor\`emes 1, 3, 4]{Lust1986} or \cite[Theorem 0.1]{LustPi1991}, we obtain the following estimate for $\bigl(\mathcal{L}_{j}(x)\bigr)_{j\in J\cup\{0\}}$ in noncommutative
vector-valued $L_p$-spaces (see \cite{Pi1998} for further details). Recall that for a sequence $(x_s)$ in $L_p(\mathcal{M})$, $\| (x_s) \|_{L_p(\mathcal{M};\ell_2^c)}=\big\| \big(\sum_s|x_s|^2\big)^{\frac{1}{2}}\big\|_p$ and $\| (x_s) \|_{L_p(\mathcal{M};\ell_2^r)}=\big\| \big(\sum_s|x_s^*|^2\big)^{\frac{1}{2}}\big\|_p$. The mixed column-row norm is 
$$
\| (x_s) \|_{L_p(\mathcal{M};\ell_2^{cr})}= \begin{cases}
\max\left\{
\left\|(x_s)\right\|_{L_p(\mathcal{M};\ell_2^{\,c})},
\left\|(x_s)\right\|_{L_p(\mathcal{M};\ell_2^{\,r})}
\right\}
\quad \text{if } 2\le p<\infty;\\
\underset{y_s+z_s=x_s}{\inf}
\left(
\|(y_s)\|_{L_p(\mathcal{M};\ell_2^{\,c})}
+
\|(z_s)\|_{L_p(\mathcal{M};\ell_2^{\,r})}
\right),
\quad \text{if } 0<p<2.
\end{cases}
$$
\begin{cor}
Keep the assumptions of \ref{Thm B} and let $c_p>0$ be the constant appearing in \ref{Thm B} which gives the norm bound for $H_\varepsilon$. Then we have
\[
(\alpha_p c_p)^{-1} \|x\|_p
\leq
\left\|
\bigl( \mathcal{L}_{j}(x)\bigr)_{j\in J\cup \{0\}}
\right\|_{L_p(\mathcal{M}_\Gamma;\ell_2^{cr})}
\leq
\beta_p c_p \|x\|_p ,
\]
where $\alpha_p$ and $\beta_p$ are the constants appearing in the lower
and upper estimates of the noncommutative Khintchine inequality. Similar estimates also apply to $\bigl( \mathcal{R}_{j}(x)\bigr)_{j\in J\cup \{0\}}$. 
\end{cor}



We now discuss some examples that fit into the context of \ref{Thm B} in the following three subsections.


\subsection{ Graph products of groups}\label{subsection;group vnA}
In this subsection, we specialize \(\mathcal{M}_{\Gamma}\) to the graph product of group von Neumann algebras \(\mathcal{L}G_{\Gamma}\), where \(\Gamma\) is a simplicial graph satisfying \eqref{eq; Fi1}. In this case, the Haagerup-type inequality in Proposition \ref{thm;equivalence-Pd-Haagerup-type}(iii) implies the property of rapid decay (property (RD)) of $G_{\Gamma}$, first introduced by Jolissaint in \cite{jolissaint1990rapidly}.

\begin{defn}
    Let $G$ be a discrete group. We say that $G$ has property (RD) with respect to a length function $L$ on $G$ if there exists a polynomial $\psi$ such that for any $d\geq 1$ and any $x\in \mathbb{C}G$ supported on group elements of length $d$, 
    $$\|x\|\le \psi(d)\|x\|_2.$$
    We say $G$ has property (RD) if there exists a length function $L$ on $G$ with respect to which $G$ has property (RD).
\end{defn}

The graph product $G_\Gamma$ 
has property (RD) with respect to the block length $\ell$  is equivalent to saying that  
there exists a polynomial $\psi$ such that for every $d\geq 1$,
	 $$\|P_{d}(x)\|_\infty\le \psi(d)\|P_{d}(x)\|_2\le \psi(d)\|x\|_2, \quad \forall x\in \mathcal{L}G_{\Gamma}.$$
By density, $P_{d}$ extends to a bounded operator from $L_2(\mathcal{L}G_{\Gamma})$ to $\mathcal{L}G_{\Gamma}$. 
Then by Proposition \ref{thm;equivalence-Pd-Haagerup-type}, 
it is equivalent to the requirement that $G_\Gamma$ be a graph product of finite groups 
satisfying 
\begin{equation}\label{eq;group-cardinality-finite}
    \sup_{s\in V\Gamma}\{n_s: n_s =|G_s|\}<\infty.
\end{equation} 
Therefore, under the assumption that $G_\Gamma$ has property (RD) with respect to the block length $\ell$, or under the assumption \eqref{eq;group-cardinality-finite} on $G_\Gamma$, we apply \ref{Thm B} to the associated group von Neumann algebra $\mathcal{L}G_\Gamma$ to obtain the boundedness of Hilbert transforms on $L_p(\mathcal{L}G_\Gamma)$ for all $1<p<\infty$.
\medskip

\subsection{Right-angled Hecke von Neumann algebras}\label{subsection;Hecke}

In this subsection, we apply \ref{Thm B} to right-angled Hecke von Neumann algebras that are graph products of 2-dimensional von Neumann algebras (see {\cite[Corollary 3.4]{caspers2020absence}}). 
We first recall some basic definitions. 

Let $\Gamma$ be a simplicial graph and $W_{\Gamma}$ the associated right-angled Coxeter group. 
Denote by $\mathbb{R}^{(W_{\Gamma},V\Gamma)}_{>0}$ (resp. $\mathbb{C}^{(W_{\Gamma},V\Gamma)}$) the set of  tuples $\mathbf{q}:= (q_s)_{s\in V\Gamma}$ in $\mathbb{R}^{V\Gamma}_{>0}$ (resp. $\mathbb{C}^{V\Gamma}$) satisfying $q_s = q_t$ whenever $s$ and $t$ are conjugate in $W_{\Gamma}$. Fix $\mathbf{q} = (q_s)_{s\in V\Gamma}\in \mathbb{R}^{(W_{\Gamma}, V\Gamma)}_{>0}$. For $s\in V\Gamma$ and $w\in W_{\Gamma}$ set $$
p_s: = \frac{q_s-1}{\sqrt{q_s}},
$$
and for $w\in W_{\Gamma}$ with reduced expression $w = s_1\ldots s_n$, define
$$\mathbf{q}_w := q_{s_1}\ldots q_{s_n}.$$
\begin{defn}[{\cite[Proposition 19.1.1]{davis2012geometry}}]
  The \textit{right-angled Hecke algebra} associated with  $(W_{\Gamma}, V\Gamma)$ and the multiparameter $\mathbf{q}$ is the unique $*$-algebra $\mathbb{C}_{\mathbf{q}}[W_{\Gamma}]$ with a linear basis $\{T_w^{(\mathbf{q})}: w\in W_{\Gamma}\}$ satisfying that for 
  $s\in V\Gamma$ and $w\in W_{\Gamma}$,  
    \begin{equation}\label{eq;Hecke-alg}
  (T_w^{(\mathbf{q})})^* = T_{w^{-1}}^{(\mathbf{q})} \quad \text{and}\quad       T_s^{(\mathbf{q})}T_w^{(\mathbf{q})} = \left\{
\begin{array}{lr}
T_{sw}^{(\mathbf{q})} & \text{if } \ell(sw)>\ell(w),\\
T_{sw}^{(\mathbf{q})} + p_s T_{w}^{(\mathbf{q})} & \text{if } \ell(sw)<\ell(w).
\end{array} \right.
    \end{equation}
\end{defn}
In particular, relation (\ref{eq;Hecke-alg}) implies that 
$$(T_{s}^{(\mathbf{q})}- q_s^{1/2})(T_{s}^{(\mathbf{q})}+ q_s^{-1/2}) = 0,\quad s\in V\Gamma.$$
For any reduced expression $s_1\ldots s_n$ of $w\in W_{\Gamma}$, one has 
$$T_w^{(\mathbf{q})} = T_{s_1}^{(\mathbf{q})}\ldots T_{s_n}^{(\mathbf{q})},$$
and this definition is independent of the choice of reduced expressions (see \cite[Lemma 19.1.2]{davis2012geometry}).

The study of general Hecke von Neumann algebras was initiated by Dymara \cite{Dym06}.
Let $\ell^2(W_{\Gamma})$ be the Hilbert space  with the canonical orthonormal basis $\{\delta_w: w\in W_{\Gamma}\}$. The right-angled Hecke algebra $\mathbb{C}_{\mathbf{q}}[W_{\Gamma}]$ admits a faithful  $*$-representation $\pi_{\mathbf{q}}: \mathbb{C}_{\mathbf{q}}[W_{\Gamma}]\to B\big(\ell^2(W_{\Gamma})\big)$ given by the left multiplication defined as follows:
$$\pi_{\mathbf{q}} (T_{s}^{\mathbf{(q)}})\delta_w = \left\{
\begin{array}{lr}
\delta_{sw} & \text{if } \ell(sw)>\ell(w),\\
\delta_{sw} + p_s \delta_{w} & \text{if } \ell(sw)<\ell(w).
\end{array} \right.$$

\begin{defn}(\cite{davis2012geometry} and \cite{Dym06})
    The \textit{right-angled Hecke von Neumann algebra} $\mathcal{M}_{\mathbf{q}}(W_{\Gamma})$ associated with $(W_{\Gamma}, V\Gamma)$ and $\mathbf{q}$ is defined by
    $$\mathcal{M}_{\mathbf{q}}(W_{\Gamma}):= \pi_{\mathbf{q}}(\mathbb{C}_{\mathbf{q}}[W_{\Gamma}])^{''}.$$
    The \textit{reduced right-angled Hecke $C^*$-algebra} $C^*_{r,\mathbf{q}}(W_{\Gamma})$ is defined by the norm closure of $\pi_{\mathbf{q}}(\mathbb{C}_{\mathbf{q}}[W_{\Gamma}])$. 
    We identify $\mathcal{M}_{\mathbf{q}}(W_{\Gamma})$ with its image under $\pi_{\mathbf{q}}$ in $B\big(\ell^2(W_{\Gamma})\big)$.
\end{defn}
There exists a canonical normal faithful tracial state $\varphi_{\Gamma,\mathbf{q}}$ on $\mathcal{M}_{\mathbf{q}}(W_{\Gamma})$ such that  $\varphi_{\Gamma,\mathbf{q}}(T^{\mathbf{(q)}}_{w}) = 0$ for any $e\ne w \in W_{\Gamma}$ (see \cite[Section 1.9]{caspers2021graph}). 
If $q_s = 1$ for all $s\in V\Gamma$, the associated Hecke von Neumann algebra  $\mathcal{M}_{\mathbf{q}}(W_{\Gamma})$ coincides with the group von Neumann algebra $\mathcal{L}W_{\Gamma}$. In this sense, Hecke von Neumann algebras can be viewed as 
$\mathbf{q}$-deformations of group algebras.  In general, however, $\mathcal{M}_{\mathbf{q}}(W_{\Gamma})\neq \mathcal{L}W_{\Gamma}$. For right-angled Hecke von Neumann algebras associated with irreducible Coxeter systems and single-parameter deformations, factoriality results are available which 
characterize precisely when they are factors (in particular, $\mathrm{II}_1$ factors); see \cite{garncarek2016factoriality}.

\medskip 



Applying Lemma \ref{prop;finite-dim-Pd-2-infty}  to right-angled Hecke von Neumann algebras associated with simplicial graphs satisfying (\ref{eq; Fi1}), we obtain the following proposition, showing that they satisfy the Haagerup-type inequality.  
This 
extends \cite[Theorem 3.4]{caspers2021graph}, where only finite simplicial graphs are considered. Moreover, our approach provides an alternative proof that does not use the Khintchine inequalities employed in \cite{caspers2021graph}. We state this result below.

\begin{prop}\label{prop;RD-Hecke-vnA}
Let $\Gamma$ be a simplicial graph satisfying (\ref{eq; Fi1}), $W_{\Gamma}$ the associated right-angled Coxeter group, $\mathbf{q} = (q_s)_{s\in V\Gamma}\in \mathbb{R}_{>0}^{(W_{\Gamma}, V\Gamma)}$ with $0<\inf_{s\in V\Gamma} q_s\le\sup_{s\in V\Gamma} q_s<\infty$, and $\mathcal{M}_{\mathbf{q}}(W_{\Gamma})$ the corresponding right-angled Hecke von Neumann algebra. Let $Q=\sup_{s\in V\Gamma}\max\{q_s,q_s^{-1}\}.$
    Then there exists a polynomial $\psi$ depending on $N$ and $Q$ 
    such that, for all $d\geq 1$ and $x\in\mathcal{M}_{\mathbf{q}}(W_{\Gamma})$, 
    $$\|P_{d} (x)\|\le \psi(d) \|x\|_2.$$    
\end{prop}
\begin{proof}
    To apply Lemma \ref{prop;finite-dim-Pd-2-infty}, it suffices to show \eqref{eq;M-finite-dimentional-condition}. Let $s\in V\Gamma$ and $x = \alpha_1 +\alpha_2 T_s^{(\mathbf{q})}\in \mathcal{M}_s = \{\mathbf{1}, T_s^{(\mathbf{q})}\}''$, where $\alpha_1,\alpha_2 \in \mathbb{C}$. Then $\|x\| = \max\{|\alpha_1+\alpha_2q_s^{1/2}|,|\alpha_1-\alpha_2q_s^{-1/2}|\}$. Therefore, the Cauchy-Schwarz inequality implies that
    $$\|x\|\le \sqrt{1+ Q}\|x\|_2,$$
    and then  \eqref{eq;M-finite-dimentional-condition} follows.
\end{proof}


Therefore, right-angled Hecke von Neumann algebras satisfy the Haagerup-type inequality in Proposition \ref{thm;equivalence-Pd-Haagerup-type} (iii).
An application of \ref{Thm B} yields that Hilbert transforms are bounded on $L_p(\mathcal{M}_{\mathbf{q}}(W_{\Gamma}))$ for all $1<p<\infty$ and all tuples $\mathbf{q}$ satisfying $0<\inf_{s\in V\Gamma} q_s\le\sup_{s\in V\Gamma} q_s<\infty$.


\subsection{Graph products of compact quantum groups}\label{sec:quantum-groups}

In this subsection, we apply \ref{Thm B} to graph products of compact quantum groups. We begin by recalling some preliminaries on compact quantum groups in the sense of \cite{KuVaes03}.
Let $\mathcal{M}$ be a finite von Neumann algebra equipped with
a comultiplication, that is, a unital normal $*$-homomorphism $\Delta: \mathcal{M}\to \mathcal{M}\bar\otimes \mathcal{M}$ satisfying the coassociativity relation
$$
(\Delta\otimes \mathrm{id})\Delta = (\mathrm{id} \otimes \Delta)\Delta.
$$
 Assume moreover that there exists a normal faithful tracial state $\varphi$ on $\mathcal{M}$
 such that 
 $$(\varphi\otimes \mathrm{id})\Delta(x)  = (\mathrm{id}\otimes \varphi)\Delta(x)=\varphi(x)1_{\mathcal{M}}, \quad \text{ for all } x\in \mathcal{M}^+.$$ 
Associated with $\varphi$, one defines an antipode $\kappa$, which is a densely defined anti-automorphism on $\mathcal{M}$ satisfying the identity
$$
(\mathrm{id}\otimes \varphi)\big((1_{\mathcal{M}}\otimes x^*)\Delta(y)\big)=\kappa \Big((\mathrm{id}\otimes \varphi)\big(\Delta(x^*)(1_{\mathcal{M}}\otimes y)\big)\Big),
$$
for all $x,y\in \mathcal{M}$. Then the triple $\mathbb{G}=(\mathcal{M}, \Delta, \varphi)$ is called a (von Neumann
algebraic) compact quantum group of Kac type.

A (finite-dimensional) unitary representation  of $\mathbb{G}$ is a unitary $U\in \mathcal{M}\bar\otimes \mathbb{M}_{n}$ such that 
$$(\Delta \otimes \mathrm{id})(U) = \big((\Sigma \otimes \mathrm{id} )(1\otimes U)\big)(1\otimes U)\in \mathcal{M}\bar\otimes \mathcal{M}\bar\otimes \mathbb{M}_{n},$$ 
where $\Sigma$ is the flip map on $\mathcal{M}\bar\otimes \mathcal{M}$. We denote by $\mathrm{Irr}(\mathbb{G})$ the set of equivalence classes of irreducible unitary representations of $\mathbb{G}$. For each $\alpha \in \mathrm{Irr}(\mathbb{G})$, fix a representative $U_\alpha\in \mathcal{M}\bar \otimes \mathbb{M}_{n_\alpha}$. The integer $n_\alpha$ is called the dimension of $\alpha$.
The discrete dual quantum group $\widehat{\mathbb{G}} = (\widehat{\mathcal{M}},\widehat{\Delta})$ is defined as follows. The von Neumann algebra $\widehat{\mathcal{M}}$  
is given by the direct sum 
$$\widehat{\mathcal{M}} = \oplus_{\alpha\in \mathrm{Irr}(\mathbb{G})} \mathbb{M}_{n_{\alpha}}.$$ 
The dual comultiplication
$
\widehat{\Delta} : \widehat{\mathcal{M}} \to \widehat{\mathcal{M}} \,\bar{\otimes}\, \widehat{\mathcal{M}}
$
is defined by
$$
\widehat{\Delta}(y) = W (y \otimes 1) W^*, 
\quad y \in \widehat{\mathcal{M}},
$$
where $W$ is the multiplicative unitary associated with $\mathbb{G}$, given by $W=\oplus_{\alpha\in \mathrm{Irr}(\mathbb{G})} U_\alpha$.
The counit $\hat{\varepsilon}:\widehat{\mathcal{M}}\to \mathbb{C}$ is the nondegenerate $*$-homomorphism satisfying $(\hat{\varepsilon}\otimes \mathrm{id})\widehat{\Delta} = (\mathrm{id}\otimes\hat{\varepsilon})\widehat{\Delta}=\mathrm{id}$.
The dual antipode $\hat{\kappa}:\widehat{\mathcal{M}}\to \widehat{\mathcal{M}}$ is a densely defined anti-automorphism such that 
$(\mathrm{id} \otimes \widehat{\kappa})(W) = W^*$. In the case where $\mathbb{G}$ is of Kac type, $\widehat{\mathbb{G}}$ admits a tracial weight given by $\hat{\varphi}= \oplus_{\alpha\in \mathrm{Irr}(\mathbb{G})}n_{\alpha}\mathrm{Tr}_{n_{\alpha}}$, where $\mathrm{Tr}_{n_{\alpha}}$ is the normalized trace on $\mathbb{M}_{n_{\alpha}}$. Finally,  $\mathbb{G}$ and its dual $\widehat{\mathbb{G}}$ are linked by the Fourier transform $\mathcal{F}: \widehat{\mathbb{G}}\to \mathbb{G}$
defined by 
$$\mathcal{F}(x) := \sum_{\alpha\in \mathrm{Irr}(\mathbb{G})}(\mathrm{id} \otimes \hat{\varphi})\big(U_{\alpha}(1\otimes x_{\alpha})\big),\quad \forall x = \oplus_{\alpha\in \mathrm{Irr}(\mathbb{G})}x_{\alpha}\in \widehat{\mathcal{M}}.$$

We  next recall the notions of central lengths and property (RD) for discrete quantum groups.
Let $\mathbb{G} =(\mathcal{M},\Delta, \varphi)$ be a compact quantum group of Kac type with $\mathcal{M}\subseteq B(\mathcal{H})$, and denote by $\widehat{\mathbb{G}} = (\widehat{\mathcal{M}},\widehat{\Delta})$ its dual discrete quantum group. The $*$-algebra
 of affiliated (unbounded) operators with $\widehat{\mathcal{M}}$ can be identified with the algebraic product $\prod_{\alpha\in \mathrm{Irr}(\mathbb{G})} \mathbb{M}_{n_{\alpha}}$. Moreover, for every affiliated operator, the subspace $\mathcal{H}_{\rm Pol}\subseteq \mathcal{H}$ consisting of 
 matrix coefficients of finite-dimensional representations of $\mathbb{G}$, 
 is a core. 
    

\begin{defn}
    Let  $\widehat{\mathbb{G}} = (\widehat{\mathcal{M}},\widehat{\Delta})$ be a  discrete quantum group. A length $L$ on $\widehat{\mathbb{G}}$ is an (unbounded) operator  affiliated with $\widehat{\mathcal{M}}$ such that $L\ge 0$, $\hat{\varepsilon}(L) = 0$, $\hat{\kappa}(L)|_{\mathcal{H}_{\rm Pol}} = L|_{\mathcal{H}_{\rm Pol}}$ and $\widehat{\Delta}(L)\le 1\otimes L + L\otimes 1$.
    For $n\geq 1$, denote by $q_n$ the spectral projection of $L$ associated with the interval $[n,n+1)$. The length $L$ is called  \textit{central} if each $q_n$ is central in the multiplier algebra of $\widehat{\mathcal{M}}$.
\end{defn}

\begin{defn} (\cite[Proposition and Definition 3.5]{vergnioux2007property})
    Let $L$ be a central length on a discrete quantum group $\widehat{\mathbb{G}} = (\widehat{\mathcal{M}},\widehat{\Delta})$. We say that $(\widehat{\mathbb{G}},L)$ has property (RD) if there exists a polynomial $\psi$ such that, for every $n\geq 1$ and every $a\in q_n\widehat{\mathcal{M}}$, one has $\|\mathcal{F}(a)\|\le \psi(n)\|a\|_2$.
\end{defn}
\begin{rem}
    Note that the terminology of property (RD) for a discrete quantum group $(\widehat{\mathbb{G}},L)$ differs from that of the Haagerup-type inequality satisfied by $(\widehat{\mathbb{G}},L)$. The latter, as defined at the beginning of this section, means that $P_d$ is bounded from $L_2$ to $L_\infty$ for all $d\ge 1$.
\end{rem}

Now, we introduce graph products of compact quantum groups and their dual discrete quantum groups.
Let $\Gamma$ be a finite simplicial graph. For  each $s\in V\Gamma$, let $\mathbb{G}_s = (\mathcal{M}_s,\Delta_s, \varphi_s)$ be a compact quantum group, where $\varphi_s$ is a normal faithful tracial state on the finite von Neumann algebra $\mathcal{M}_s$. 
The graph product of the family $\{\mathbb{G}_s:s\in V\Gamma\}$ is again a compact quantum group $\mathbb{G} = (\mathcal{M}_{\Gamma},\Delta, \varphi_\Gamma)$ of Kac type (see \cite[Theorem 4.4]{caspers2017graph}). Here $\mathcal{M}_{\Gamma}$ is the graph product of the family $\{\mathcal{M}_s:s\in V\Gamma\}$, $\varphi_\Gamma$ is the associated graph product state, and the comultiplication $\Delta:\mathcal{M}_{\Gamma}\to \mathcal{M}_{\Gamma}\bar\otimes \mathcal{M}_{\Gamma}$ is  determined by the condition $\Delta|_{\mathcal{M}_s} = \Delta_s$ for every $s\in V\Gamma$. 

Denote by $\widehat{\mathbb{G}}_s = (\widehat{\mathcal{M}}_s, \widehat{\Delta}_s)$ the  dual discrete quantum group of $\mathbb{G}_s$. The dual discrete quantum group $\widehat{\mathbb{G}} = (\widehat{\mathcal{M}}_{\Gamma}, \widehat{\Delta})$ associated with $\mathbb{G}$ is the direct sum of matrix algebras
    $$\widehat{\mathcal{M}}_{\Gamma}=\oplus_{k\geq 0}\widehat{\mathcal{M}}_{k},$$
  where   
    $\widehat{\mathcal{M}}_{k}={\underset{\alpha = \alpha_1\otimes\ldots\otimes \alpha_k}{\underset{\alpha\in \mathrm{Irr}(\mathbb{G})}{\oplus}}}\mathbb{M}_{n_{\alpha_1}}\otimes\ldots\otimes \mathbb{M}_{n_{\alpha_k}}$.
    For $k =0$, the index $\alpha$ corresponds to the trivial representation $\mathbf{1}$.


\medskip


We now state the main proposition of this subsection, which asserts that $P_{d}$ is bounded from $L_2(\mathbb{G})$ to $\mathbb{G}$ for all $d\geq 1$. 
As explained in the paragraph above Lemma \ref{prop;Equi-char-finite-dim-RD-1},
the graph products of compact quantum groups satisfying the assumptions in the following proposition provide examples of graph products of finite-dimensional von Neumann algebras. 

\begin{prop}\label{prop;Pd-L1-Lq-bound-quantum-group}
Let $\Gamma$, $\mathbb{G} = (\mathcal{M}_{\Gamma},\Delta)$ and  $\widehat{\mathbb{G}} = (\widehat{\mathcal{M}}_{\Gamma},\widehat{\Delta})$ be as above. Assume that $(\widehat{\mathbb{G}}_s, L_s)$ has property (RD) for every $s\in V\Gamma$, and that 
    $\widehat{\mathbb{G}}_{\Gamma_0}$ has property (RD) for every $\Gamma_0\in \mathrm{Cliq}(\Gamma)$. 
The following are equivalent.
\begin{enumerate}
    \item [(i)](\cite[Proposition 5.7]{caspers2017graph}) There exists a polynomial $\psi$ such that for all $k,l,m\geq 1$ with $|k-l|\le m\le k+l$ and $x\in \widehat{\mathcal{M}}_k$, we have $\|P_m\mathcal{F}(x)P_l\|\le \psi(k)\|x\|_2$.

    \item [(ii)] There exists a polynomial $\psi$ such that for all $k\geq 1$ and all $x\in \widehat{\mathcal{M}}_k$, $\|\mathcal{F}(x)\|\le \psi(k)\|x\|_2$.

    \item[(iii)] There exists a polynomial $\psi$ such that for all $k\geq 1$ and all $y\in \mathcal{M}_{\Gamma}$, $\|P_k(y)\|\le \psi(k)\|y\|_2$.
\end{enumerate}    
\end{prop}
\begin{proof}
    The implication (ii)$\implies$(i) is immediate. 
    The equivalence between (ii) and (iii) follows from the fact that $\mathcal{F}$ is a unitary map from $L_2(\widehat{\mathcal{M}}_k)$ onto $ L_2({\mathcal{M}}_k)$ and intertwines the homogeneous decompositions, while $P_k$ is an $L_2$-orthogonal projection onto the $k$-th homogeneous subspace.
    The implication (i)$\implies$(iii) follows immediately by applying Lemma \ref{prop;finite-dim-Pd-2-infty} to $\mathbb{G}$. 
\end{proof}

By Proposition \ref{thm;equivalence-Pd-Haagerup-type},  we conclude that 
under the assumptions of the preceding proposition, the boundedness of Hilbert transforms on $L_p(\mathbb{G})$ for all $1<p<\infty$ follows from \ref{Thm B}. 


\section{Application to Ozawa's compactness problems}\label{sec; Ozawa-application}
In \cite{ozawa2010comment}, Ozawa asked to what extent the $C^*$-algebraic techniques used in the proof of solidity for free group factors can be implemented directly at the level of von Neumann algebras. In particular, he posed the following question in \cite[Problem]{ozawa2010comment}.
Let $\mathbb{F}_n$ denote the free group with $n$ generators. For  $h\in \mathbb{F}_n$, let $R_h$ be the projection from $L_2(\mathcal{L}\mathbb{F}_n)$ onto the closed subspace generated by those $\lambda_g$ for which $g$ admits a factorization $g = g'h$ such that $\ell(g) = \ell(g') + \ell(h)$. 
 For $n\ge 2$, does the commutator $[R_h,\mathcal{L}\mathbb{F}_n]$ map the unit ball of $\mathcal{L}\mathbb{F}_n$ to a compact set in $L_2(\mathcal{L}\mathbb{F}_n)$ for every $h\in \mathbb{F}_n$?

Mei and Ricard answered this question affirmatively in \cite[Corollary 4.10]{mei2017free} by introducing a Hilbert transform $H^{Rd}_{\varepsilon}$ associated with a fixed length $d$, which satisfies the required compactness property. In this section, we give a complete characterization of this question for graph products of finite groups and right-angled Hecke von Neumann algebras. In both settings, the problem has a positive answer if and only if every connected component of $\Gamma$ is a clique.


Throughout this section, let $\Gamma$ be a finite simplicial graph, and let $\mathcal{M}_{\Gamma}$ denote either the group von Neumann algebra $\mathcal{L}G_\Gamma$ associated with a graph product of nontrivial finite groups, or the right-angled Hecke von Neumann algebra $\mathcal{M}_{\mathbf{q}}(W_\Gamma)$ associated with a multiparameter $\mathbf{q} = (q_s)_{s\in V\Gamma}$ in $\mathbb{R}^{V\Gamma}_{>0}$. We introduce a new Hilbert transform that generalizes the transform $H_{\varepsilon}^{Rd}$ from the free group setting \cite{mei2017free}.

Let $\mathcal{B}$ denote the canonical normalized orthonormal basis of
$L_2(\mathcal M_\Gamma)$:
$$\mathcal B=
\begin{cases}
\{\lambda_g:g\in G_\Gamma\},&\text{if }\mathcal M_\Gamma=\mathcal L G_\Gamma,\\
\{T_w^{(\mathbf q)}:w\in W_\Gamma\},
&\text{if }\mathcal M_\Gamma=\mathcal M_{\mathbf q}(W_\Gamma).
\end{cases}$$

\begin{defn}\label{def; d Hilbert}
	Given $d\geq 1$, denote by $\mathcal{B}_d$ the set of basis elements of $L_2(\mathcal{M}_{\Gamma})$ of length $d$. For $h\in \mathcal{B}_{d}$, define the following projections. Recall that $\mathcal{F}_{\Gamma}$ is the linear span of $\mathbf{1}$ and all reduced operators of $\mathcal{M}_{\Gamma}$.
	\begin{enumerate}
		\item[(i)] Let $\mathcal{L}_h$ denote the projection from $\mathcal{F}_{\Gamma}$ onto the subspace $S_h$ generated by those $a\in\mathcal{B}$ satisfying $\ell(a)\ge d$ and such that, for every factorization \[
a=a'a'',
\qquad
\ell(a)=\ell(a')+\ell(a''),
\qquad
\ell(a')=d,
\]
one necessarily has $a' = h$. 
		
		\vspace{0.1cm}
		
		\item[(ii)] Let $\mathcal{L}_{C,d}$ denote the projection from $\mathcal{F}_{\Gamma}$ onto the subspace $S_C$ generated by those $a\in \mathcal{B}$ satisfying $\ell(a)\ge d$ and admitting two factorizations $a=a'a{''}=b'b{''}$ with $\ell(a') = \ell(b') = d$ but $a'\ne b'$.
		
		\vspace{0.1cm}
		
		\item[(iii)] Let $\mathcal{R}_h$ (resp. $\mathcal{R}_{C,d}$) denote  the projection from $\mathcal{F}_{\Gamma}$ onto the subspace $S_h^*$ (resp. $S_C^*$).

        \vspace{0.1cm}
        
		\item[(iv)] Define the Hilbert transforms $H_{\varepsilon}^{\mathcal{L},d}$ and $H_{\varepsilon}^{\mathcal{R},d}$ on $\mathcal{F}_{\Gamma}$ by
		$$H_{\varepsilon}^{\mathcal{L},d} := \varepsilon_0 P_{\le d-1} + \sum_{h\in \mathcal{B}_{d}} \varepsilon_h\mathcal{L}_h + {\varepsilon}_d \mathcal{L}_{C,d}\quad \text{and} \quad H_{\varepsilon}^{\mathcal{R},d} := \varepsilon_0 P_{\le d-1} + \sum_{h\in \mathcal{B}_{d}} \varepsilon_{h^{*}}\mathcal{R}_h + {\varepsilon}_d \mathcal{R}_{C,d},$$
		where $\varepsilon = (\varepsilon_0,\varepsilon_h, {\varepsilon}_d)_{h\in \mathcal{B}_{d}}$ with $\varepsilon_0, \varepsilon_h, {\varepsilon}_d\in \{\pm1\}$.
	\end{enumerate}
\end{defn}


\begin{proof}[Proof of \ref{Thm C}]
Since $\mathcal{B}_d$ is finite and $P_{\le d-1}$ is $L_p$-bounded by Proposition \ref{thm;equivalence-Pd-Haagerup-type}, it suffices to prove that $\mathcal{L}_h$ is $L_p$-bounded for every $h\in \mathcal{B}_d$. Fix $h\in \mathcal{B}_d$ and write $h = h_1\ldots h_d$ as a reduced operator of type $w = s_1\ldots s_d$.
Define $I_\mathrm{init}(h) = \{s_i: 1\le i\le d,   s_j\in \mathrm{Link}(s_i) \text{ for every } j<i\}$ and
$I_{\mathrm{last}}(h) = \{s_i: 1\le i\le d, s_j\in {\rm{Link}}(s_i) \text{ for every } j>i\}$, the set of vertices that correspond to the possible initial and last syllables of $h$, respectively. For $s\in V\Gamma$, let $\mathrm{Star}(s) = \{s\}\cup \mathrm{Link}(s)$ and set $\mathcal{S}(h) = \cup_{s\in I_{\mathrm{last}(h)}}\mathrm{Star}(s)$. Put $M_h := \prod_{t\in \mathcal{S}(h)} {L}_t^\perp$. We claim that
    $$\mathcal{L}_h = h M_h h^*.$$
If $\mathcal{M}_{\Gamma} = \mathcal{L}G_\Gamma$, we verify the identity on the canonical basis. Let $b = \lambda_g$ satisfy $\mathcal{L}_h(b) = b$. Then there is $g'\in G_\Gamma$ such that $b= h\lambda_{g'}$ and then $\lambda_{g'} = h^* b$. By the definition of $\mathcal{L}_h$, every element in $I_{\mathrm{last}(h)}$ is free to every element in $I_{\mathrm{init}(\lambda_{g'})}$. Therefore, we have 
$$I_{\mathrm{init}(\lambda_{g'})}\cap \mathcal{S}(h) = \emptyset\quad \text{and}\quad M_h(\lambda_{g'}) = \lambda_{g'}.$$
Consequently, $\mathcal{L}_h(b) = h M_h h^*(b)$. If $b= \lambda_g$ is such that $\mathcal{L}_h(b) = 0$, it is easy to see that we also have $h M_h h^*(b)=0$.
When $\mathcal{M}_\Gamma = \mathcal{M}_{\mathbf{q}}(W_\Gamma)$, using the defining multiplication relations of the right-angled Hecke von Neumann algebra, an induction on the length of $h$ yields the identity.

We now show that $\mathcal{L}_h$ is bounded on $L_p(\mathcal{M}_\Gamma)$ for every $h\in \mathcal{B}_d$. Let $C_p$ be a uniform bound for the operators $L_s^\perp$ on $L_p(\mathcal{M}_\Gamma)$, for all $s\in V\Gamma$, as provided by Theorem~\ref{thm;finite-graph-Lp-Hilbert-transform}. It then follows from H\"older's inequality  that 
$$\|\mathcal{L}_h(x)\|_p = \|h M_h h^*(x)\|_p\le \|h\|_{\mathcal{M}_\Gamma}^2C_p^{\# \mathcal{S}(h)}\|x\|_p\le \|h\|_{\mathcal{M}_\Gamma}^2C_p^{|V\Gamma|}\|x\|_p.$$
If $\mathcal{M}_{\Gamma} = \mathcal{L}G_\Gamma$, then $\|h\|_{\mathcal{M}_\Gamma} = 1$. 
    If $\mathcal{M}_{\Gamma} = \mathcal{M}_{\mathbf{q}}(W_\Gamma)$, then $h = T_{s_1}^{(\mathbf{q})}\ldots T_{s_d}^{(\mathbf{q})}$. Therefore,
    $$\|h\|_{\mathcal{M}_\Gamma}\le D_{\mathbf{q},h}:= \prod_{i=1}^d\max\{q_{s_i}^{1/2}, q_{s_i}^{-1/2}\}.$$
    In either case, $\mathcal{L}_h$ is bounded on $L_p(\mathcal{M}_\Gamma)$.
\end{proof}

We now give an equivalent characterization of Ozawa's compactness problem for $\mathcal{M}_\Gamma$. 

\begin{thm}\label{cor; Ozawa's problem}
    Let  $d \geq 1$ and $p > 2$. The following are equivalent:

\begin{enumerate}
\item[(i)] Every connected component of $\Gamma$ is a clique.

\item[(ii)] For every $\varepsilon$ and every $x \in C_r^*(G_{\Gamma})$, one has
\[
[H_{\varepsilon}^{\mathcal{R},d}, x] \in \mathbb{K}(\ell_2 G_{\Gamma}),
\]
and for every $x \in C_{r,\mathbf{q}}^*(W_{\Gamma})$,
\[
[H_{\varepsilon}^{\mathcal{R},d}, x] \in \mathbb{K}(\ell_2 W_{\Gamma}).
\]

\item[(iii)] For every $\varepsilon$ and every $x \in L_p(\mathcal{L}G_{\Gamma})$ (resp. $x \in L_p\big(\mathcal{M}_{\mathbf{q}}(W_{\Gamma})\big)$), the operator $[H_{\varepsilon}^{\mathcal{R},d}, x]$ maps the closed unit ball of $\mathcal{L}G_{\Gamma}$ (resp. $\mathcal{M}_{\mathbf{q}}(W_{\Gamma})$) into a compact subset of $\ell_2 G_{\Gamma}$ (resp. $\ell_2 W_{\Gamma}$). In particular, this holds for all $x \in \mathcal{L}G_{\Gamma}$ (resp. $x \in \mathcal{M}_{\mathbf{q}}(W_{\Gamma})$).
\end{enumerate}  
\end{thm}
\begin{proof}
For (i) $\implies$ (ii), it suffices to prove that $[\mathcal{R}_h, \lambda_g] \in \mathbb{K}(\ell_2 G_{\Gamma})$ and $[\mathcal{R}_{C,d}, \lambda_g] \in \mathbb{K}(\ell_2 G_{\Gamma})$ for all $g,h \in G_{\Gamma}$. Let $g,g',h \in G_{\Gamma}$. On the one hand,
\[\lambda_g \mathcal{R}_h \delta_{g'} = \lambda_g \delta_{g'} \quad \text{if } g' = h' h \text{ for some } h' \in G_{\Gamma}.\]
On the other hand, since every connected component of $\Gamma$ is a clique, direct products of finite groups are finite. So 
$G_\Gamma$ admits a free product decomposition $\ast_{i=1}^n G_i$, where $n$ is the number of connected components of $\Gamma$ and each $G_i$ is a finite group. Let $L$ denote the block length with respect to this free product decomposition. We have
\[\mathcal{R}_h \lambda_g \delta_{g'} = \lambda_g \delta_{g'} \quad \text{if } g' = g^{-1} h' h \text{ for some } h' \in G_{\Gamma}.\]
It follows that\[[\mathcal{R}_h, \lambda_g]\delta_{g'} = 0 \quad \text{whenever } L(g') \ge L(g) + L(h).\] Since each $G_i$ is finite, the set\[\{ g' \in G_{\Gamma} : L(g') < L(g) + L(h) \}\]is finite, and therefore $[\mathcal{R}_h, \lambda_g]$ has finite rank. A similar argument shows that\[[\mathcal{R}_{C,d}, \lambda_g]\delta_{g'} = 0 \quad \text{whenever } L(g') \ge L(g) + d,\]which implies that $[\mathcal{R}_{C,d}, \lambda_g]$ also has finite rank. The corresponding statements for $[\mathcal{R}_h, T_w^{(\mathbf{q})}]$ and $[\mathcal{R}_{C,d}, T_w^{(\mathbf{q})}]$ on $\ell_2 W_{\Gamma}$ follow analogously.
\medskip

For (ii)$\implies$ (iii), let $q > 2$ be such that $\frac{1}{p} + \frac{1}{q} = \frac{1}{2}$. For $x \in L_p(\mathcal{L}G_{\Gamma})$ and $y \in L_q(\mathcal{L}G_{\Gamma})$, an argument as in \cite[Section 5]{ozawa2010comment} (see also \cite[Corollary 4.10(iii)]{mei2017free}), together with H\"older's inequality and the $L_q$-boundedness of $H_{\varepsilon}^{\mathcal{R},d}$, yields\[\|[H_{\varepsilon}^{\mathcal{R},d}, x]y\|_2\le \|H_{\varepsilon}^{\mathcal{R},d}\| \, \|xy\|_2+ \|x\|_p \, \|H_{\varepsilon}^{\mathcal{R},d} y\|_q\le 2 \|H_{\varepsilon}^{\mathcal{R},d}\| \, \|x\|_p \, \|y\|_q.\]Since $C_r^*(G_{\Gamma})$ is dense in $L_p(\mathcal{L}G_{\Gamma})$ and $\mathcal{L}G_{\Gamma} \subseteq L_l(\mathcal{L}G_{\Gamma})$ for $l = p,q$, the conclusion follows.
The proof for $\mathcal{M}_{\mathbf{q}}(W_\Gamma)$ is identical to the above, replacing $\mathcal{L}G_{\Gamma}$ and $\ell_2 G_{\Gamma}$ by $\mathcal{M}_{\mathbf{q}}(W_{\Gamma})$ and $\ell_2 W_{\Gamma}$, respectively.

\medskip

For (iii) $\implies$ (i), suppose on the contrary that there is a connected component of $\Gamma$ that is not a clique. Then it contains an induced subgraph of the form shown in  Figure \ref{fig:figure2}. 
Consider first the group algebra case. Choose nontrivial elements $g_i\in G_{s_i}$, for $i=1,2,3$, and set $x = \lambda_{g_2}$ and $y_n = \lambda_{(g_1g_3)^n}$ for $n\ge 1$. Then $(y_n)_n$ is in the unit ball of $\mathcal{L}G_\Gamma$ and is an orthonormal sequence in $\ell_2G_\Gamma$. Therefore, $(y_n)_n$ converges weakly to 0 in $\ell_2G_\Gamma$.
Fix $n_0\ge d$ and let $a\in \mathcal{B}_d$ be the last $d$ syllables of $y_{n_0}$. Since $g_{1}$ and $g_{3}$ are free to each other, $a$ is the unique last $d$ syllables of $y_n$ for any $n\ge n_0$. Then $\mathcal{R}_{a^*}(y_{n}) = y_n$ for any $n\ge n_0$. However, since $g_{2}$ commutes with both $g_{1}$ and $g_{3}$, we have $\mathcal{R}_{a^*}(xy_n) = 0$ for $n\ge n_0$. Therefore, 
$$[\mathcal{R}_{a^*},x]y_n = -xy_n,\; \forall \, n\ge n_0.$$
Choose $\varepsilon_{a}=1$ and every other sign equal to $-1$. Then we have
$H_\varepsilon^{\mathcal R,d}=2\mathcal R_{a^*}-{\rm id}$, and hence
\[
\|[H_\varepsilon^{\mathcal R,d},x]y_n\|_2=2\nrightarrow 0.
\]
By \cite[Proposition]{ozawa2010comment}, $\left([H_{\varepsilon}^{\mathcal{R},d},x]y_n\right)_{n\ge n_0}$ is not compact, which contradicts (iii).

We next consider the right-angled Hecke algebra case. 
Let $\mathcal N$ be the von Neumann subalgebra generated by $T_{s_1}^{(\mathbf{q})}$ and $T_{s_3}^{(\mathbf{q})}$. 
	By the graph product construction, this is the trace-preserving
	embedded free product
	$$
	\mathcal N\cong
	\mathcal M_{q_{s_1}}(\mathbb Z_2)*\mathcal M_{q_{s_3}}(\mathbb Z_2).
	$$
	The trace weights of these two vertex algebras are
	$$
	a=\frac{1}{1+q_{s_1}},\quad 1-a,
	\qquad
	b=\frac{1}{1+q_{s_3}},\quad 1-b.
	$$
	All four weights lie in $(0,1)$. After interchanging the factors and
	relabelling their minimal projections, choose the first projection to
	have trace
	$$
	\alpha=\max\{a,1-a,b,1-b\},
	$$
	and let $\beta$ be a projection weight of the other factor. Then
	$1>\alpha\ge\max\{\beta,1-\beta\}>0$. It then follows from 
	Dykema's two-projection theorem
	\cite[Theorem~1.1]{Dykema1993} that there is a nonzero central
	projection $z\in\mathcal N$ such that
	$$
	z\mathcal N\cong
	L_\infty([0,\pi/2],\nu)\,\overline\otimes\, \mathbb M_2,
	$$
	where $\nu$ is a nonatomic probability measure. Let $c:=\tau(z)=2(1-\alpha)>0$.
	Equip $z\mathcal N$ with the normalized trace
	$\tau_z=c^{-1}\tau|_{z\mathcal N}$.
	Choose a Haar unitary $U$ in its diffuse abelian subalgebra $L_\infty([0,\pi/2],\nu)\,\overline{\otimes}\,\mathbb{C}\, \rm id_{\mathbb M_2}$. Thus
	$$
	U^*U=UU^*=z
	\;\; \text {and} \;\;
	\tau_z(U^k)=0,\quad \forall \, k\in\mathbb Z\setminus\{0\}.
	$$
	Set $y_n=U^n$, which are elements of $\mathcal N\subset\mathcal{M}_\Gamma$.
	Then $\|y_n\|=1$ and $\tau(y_m^*y_n)=c\,\delta_{mn}$.
	Consequently, $(y_n/\sqrt c)_n$ is an orthonormal sequence in
	$L_2(\mathcal{M}_\Gamma)$, and hence $y_n$ converges weakly to 0 in $L_2(\mathcal{M}_\Gamma)$.
	Let $h_1 = T_{s_1}^{(\mathbf{q})}T_{s_3}^{(\mathbf{q})}T_{s_1}^{(\mathbf{q})}\ldots$ and $h_2 = T_{s_3}^{(\mathbf{q})}T_{s_1}^{(\mathbf{q})}T_{s_3}^{(\mathbf{q})}\ldots$ be the two alternating Hecke basis words of length $d$. Since every canonical basis word of $\mathcal N$ of length at least $d$
	has a unique $d$ syllables equal to exactly one of $h_1,h_2$. Therefore,
	on $L_2(\mathcal N)$,
	$$
	\mathcal R_{h^*_1}+\mathcal R_{h^*_2}={\rm id}-P_{\le d-1}.
	$$
	The ranges of these projections are orthogonal.
	It follows that
	$$
	\|\mathcal R_{h_1^*}y_n\|_2^2+\|\mathcal R_{h_2^*}y_n\|_2^2
	=c-\|P_{\le d-1}y_n\|_2^2
	\rightarrow c.
	$$
	For sufficiently large $n$, the sum is at least $c/2$.
	Hence there exist a fixed $j\in\{1,2\}$ and a subsequence $(y_{n_k})$
	such that
	\begin{equation}\label{eq; Hecke-suffix-mass}
		\|\mathcal R_{h_j^*}y_{n_k}\|_2\ge\frac{\sqrt c}{2},
		\qquad k\ge1.
	\end{equation}
	Put $x=T_{s_2}^{(\mathbf{q})}$. By the defining multiplication relations of $\mathcal{M}_{\mathbf{q}}(W_\Gamma)$ and the fact that $s_2$ commutes with $s_1$ and $s_3$, we have
	\begin{equation}\label{eq; Hecke-suffix-annihilation}
		\mathcal R_{h_j^*}(xv)=0,\qquad v\in L_2(\mathcal N).
	\end{equation}
	So the algebra generated by $\mathcal M_{s_2}$ and $\mathcal N$ is their
	trace-preserving tensor product. Since $\varphi_{s_2}(T_{s_2}^{(\mathbf{q})*}T_{s_2}^{(\mathbf{q})})=1$, for
	every $v\in L_2(\mathcal N)$ we have
	\begin{equation}\label{eq; Hecke-middle-isometry}
		\|xv\|_2^2
		=\varphi_{s_2}(T_{s_2}^{(\mathbf{q})*}T_{s_2}^{(\mathbf{q})})\|v\|_2^2
		=\|v\|_2^2.
	\end{equation}
	Combining \eqref{eq; Hecke-suffix-mass}--\eqref{eq; Hecke-middle-isometry},
	we obtain
	$$
	[\mathcal R_{h_j^*},x]y_{n_k}=-x\mathcal R_{h_j^*}y_{n_k}
	\quad\text{and}\quad 
	\|[\mathcal R_{h_j^*},x]y_{n_k}\|_2\ge\frac{\sqrt c}{2}.
	$$
	Choosing $\varepsilon_{h_j}=1$ and every other sign equal to $-1$, we have 
	$H_\varepsilon^{\mathcal R,d}=2\mathcal R_{h_j^*}-{\rm id}$.
	Thus, 
	$$
	\|[H_\varepsilon^{\mathcal R,d},x]y_{n_k}\|_2
	=2\|[\mathcal R_{h_j^*},x]y_{n_k}\|_2
	\ge\sqrt c>0.
	$$
    This again contradicts (iii), by \cite[Proposition]{ozawa2010comment}.
	Therefore, every connected component of $\Gamma$ is a clique.
\end{proof}



\section*{Acknowledgements}
The authors are grateful to \'Eric Ricard for encouraging them to extend the Mei-Ricard theory for amalgamated free products of von Neumann algebras to the setting of graph products. We also thank him for his helpful comments, which improved the presentation of this paper.

\bibliographystyle{plain}
\bibliography{graph-product}
\end{document}